\documentclass{amsproc}
\usepackage{latexsym}
\usepackage{amssymb}
\usepackage{amsfonts}
\usepackage{amstext}
\usepackage{multicol}
\usepackage{amsthm}
\usepackage{tikz}
\usepackage{amsxtra}  
\usepackage[pdftex, breaklinks, linktocpage=true, bookmarksopen=true,bookmarksopenlevel=0,bookmarksnumbered=true]{hyperref}

\usepackage[margin=0.75in]{geometry}

\hypersetup{
pdftitle = {Research Paper},
pdfauthor = {\textcopyright\ Lanbo Fang},
pdfcreator = {\LaTeX\ with package \flqq hyperref\frqq},
colorlinks = {true}
}

\newcommand{\R}{\mathbf{R}}
\newcommand{\Z}{\mathbf{Z}}

\newtheorem{theorem}{Theorem}[section]
\newtheorem{lemma}[theorem]{Lemma}
 \newtheorem{corollary}[theorem]{Corollary}

\theoremstyle{definition}

\theoremstyle{remark}

\numberwithin{equation}{section}

\begin{document}

\title{Eigenvalue Asymptotics of Narrow Domains}

\author{Lanbo Fang}
\address{617 N Santa Rita Ave, Tucson, AZ, 85721}
\curraddr{}
\email{lfang@math.arizona.edu}
\thanks{}

\subjclass[2010]{Primary }

\date{\today}

\begin{abstract} 
In this paper, we considered the spectrum of the Dirichlet Laplacian $\Delta_\epsilon$ on $\Omega_\epsilon=\{(x,y): -l_1<x<l_2, 0<y<\epsilon h(x)]\}$ where $ l_1,l_2>0$ and  $h(x)$ is a positive analytic function having $0$ the only point where it achieves its global maximum $M$. In particular we studied in details about the full asymptotics of the eigenvalues.

\end{abstract}

\maketitle

\section{Introduction}
It is very interesting to study the spectrum of Laplace Operator on a thin domain with particular boundary conditions. A nice survey on this topic is by Daniel Grieser \cite{7}. For interesting applications in related mathematical areas see \cite{1}, \cite{2}, \cite{4}. There are also a lot of applications in mathematical physics as in \cite{9}, \cite{10}, \cite{11}. 
 For higher dimension situations see \cite{6}. For discussions on the case for Neumann boundary conditions see \cite{8}.

This paper is motivated by the work \cite{LS} in 2009 where the authors obtained a two-term asymptotics of the eigenvalues for the Dirichlet Laplacian $\Delta_\epsilon$ in a family of bounded domains $\Omega_\epsilon=\{(x,y): -l_1<x<l_2, 0<y<\epsilon h(x)]\}$ where $ l_1,l_2>0$ and $h(x)$ is a positive analytic function having $0$ the only point where it achieves its global maximum $M$. With such assumptions on $h(x)$, one easily see $h(x)=M-c(x)x^m$ for some even integer $m$ and some positive analytic function $c(x)$ with $c(0)=c_0\neq 0.$ The two term asymptotics that was found in \cite{LS} is the following:
\begin{align*}
\lim_{\epsilon\rightarrow 0} \epsilon^{2\alpha_1}\left(\Lambda_j(\epsilon)-\frac{\pi^2}{M^2\epsilon^2}\right)=\mu_j
\end{align*}
where $\mu_j$ are eigenvalues of the operator on $L^2(\mathbf{R})$ given by 
\begin{align*}
\mathbf{H}=-\frac{d^2}{dx^2}+q(x), \quad q(x)=\frac{2c_0\pi^2 }{M^3}x^m
\end{align*} and $\Lambda_j(\epsilon)$ are eigenvalues of the Dirichlet Laplacian $\Delta_\epsilon$ 
with 
$\alpha_1=\frac{2}{m+2}$. 

In this paper, we will figure out the formula for the full asymptotics of the eigenvalues $\Lambda_j(\epsilon)$ as $\epsilon\rightarrow 0.$

\section{Main Results}
To fix the notation, let $\Omega_\epsilon=\{(x,y): -l_1<x<l_2, 0<y<\epsilon h(x)]\}$ where $ l_1,l_2>0$ and $h(x)=M-c(x)x^m$ is a positive analytic function having $0$ the only point achieving the global maximum. And we have the taylor expansion $c(x)=c_0+\sum_{n=1}^\infty c_nx^n$. 


Let $\Delta_\epsilon$ be the Dirichlet Laplacian on $\Omega_\epsilon$ and  we are interested in the asymptotic behavior of the spectrum of $\Delta_\epsilon.$ More precisely, we would consider the following Dirichlet eigenvalue problem
\begin{align}
\label{1}
\Delta_\epsilon u & =\Lambda_\epsilon u 
\end{align}
It is well known that the eigenvalues $\Lambda_\epsilon$ are discrete and tends to infinity. The question we want to solve in this paper is to figure out the full asymptotics of  the eigenvalues $\Lambda_\epsilon$ as $\epsilon\rightarrow 0.$

There are three stages in the whole analysis. The starting point of the whole analysis is to restrict the Dirichlet Laplacian $\Delta_\epsilon$ to a proper closed subspace and turn the whole problem into a perturbation problem. In particular we have the following Lemma.

\begin{lemma}
The Dirichlet eigenvalue problem 
\begin{align*}
\Delta_\epsilon u &=\Lambda_\epsilon u
\end{align*}
is equivalent to 
\begin{align*}
A_{11}u_1+A_{12}u_2 &=\Lambda_\epsilon u_1 \\
A_{21}u_1+A_{22}u_2 &=\Lambda_\epsilon u_2\\
\end{align*}
where $u_1=Pu, u_2=Qu,A_{11}=P\Delta_\epsilon P, A_{12}=P\Delta_\epsilon Q, A_{21}=Q\Delta_\epsilon P, A_{22}=Q\Delta_\epsilon Q,$ with $P$ the othorgonal projection onto \begin{align*}
\mathcal{L}_\epsilon=\{u(x,y)=\chi(x)\sqrt{\frac{2}{\epsilon h(x)}}\sin(\frac{\pi y}{\epsilon h(x)}): \chi(x)\in H_0^1([-l_1, l_2])\}
\end{align*}
and $Q=\mathbf{I}-P.$
\end{lemma}

Following this Lemma 2.1 we study in details about the operator $A_{11}$ in our second stage. The key idea here is by introducing scaling $x=\epsilon^{\alpha_1}y, \alpha_1=\frac{2}{m+2} $  one will have
\begin{align*}
\epsilon^{2\alpha_1}\left(A_{11}-\frac{\pi^2}{M^2\epsilon^2}\right) \text{\quad unitarily equivalent to \quad }H_\epsilon= H_0+\sum_{n=1}^\infty H_n\epsilon^{n\alpha_1}
\end{align*}
where $H_0=-\frac{d^2}{dy^2}+\frac{2\pi^2c_0}{M^3} y^m$ is an anharmonic oscillator and $H_n$ is some polynomial in $y$ of degree $n+m$.

Using resolvent expansion one will have  the full asymptotic expansion of the eigenvalues  for the operator $H_\epsilon$ thanks to the  exponential decaying of eigenfunctions of $H_0$ and the fact that $H_\epsilon=H_0+\sum_{n=1}^\infty H_n\epsilon^{n\alpha_1}$ defined over $H_0^1([-\frac{l_1}{\epsilon^{\alpha_1}}, \frac{l_2}{\epsilon^{\alpha_1}}])$ can be  approximated in a certain sense by $H_{\epsilon,K}=H_0+\sum_{n=1}^K H_n\epsilon^{n\alpha_1}$ defined over $H_0^1(\R)$. In summary the full asymptototics of the eigenvalues $\lambda$ of the model oeprator $A_{11}$ is stated in the following Theorem.

\begin{theorem}
Let $\{\mu_j\}_{j=0}^\infty$ be the full set of eigenvalues of $H_0$ defined on $H_0^1(\R)$. Then  $\epsilon^{2\alpha_1}\left(A_{11}-\frac{\pi^2}{M^2\epsilon^2}\right)$ has full eigenvalue asymptotics  given by 
\begin{align*}
\nu&\sim\mu_j+\sum_{n=1}^\infty q_n\epsilon^{n\alpha_1}
\end{align*}
where $q_n$ can be computed explicitly.
\end{theorem}

The last stage of the work is trying to understand the difference $\tilde{\lambda}=\Lambda-\lambda$ between the eigenvalues of the original Dirichlet Laplacian operator $\Delta_\epsilon$ and the model operator $A_{11}$. The result is stated as follows.
\begin{theorem}
Let $\lambda$ be eigenvalues of $A_{11}$ with normalized eigenfunction $\phi$. We also let  $\tilde{\lambda}=\Lambda_\epsilon-\lambda$ . Then $\tilde{\lambda}_n\rightarrow \tilde{\lambda}$ as $n\rightarrow \infty,$ where
\begin{align*}
\tilde{\lambda}_0&=a_0\\
\tilde{\lambda}_{n+1}&=g(\tilde{\lambda}_n)\\
\end{align*}
 $ a_0=-\frac{\langle u_1, A_{12}(A_{22}-\lambda)^{-1}A_{21}\phi\rangle}{\langle u_1,\phi\rangle}, a_n=-\frac{\langle u_1, A_{12}(A_{22}-\lambda)^{-n-1}A_{21}\phi\rangle}{\langle u_1,\phi\rangle}$ and $g(x)=a_0+\sum_{n=1}^\infty a_nx^n$.
\end{theorem}

In this way we have a detailed analysis for the full asymptotics of the Dirichlet Eigenvalues. The following of the paper is organized in the following way. In section 3 we will prove Lemma 2.1 (in Lemma 3.1) and a more concrete version of Theorem 2.2( in Theorem 3.19). In section 4 we will prove Theorem 2.3 ( in Theorem 4.7 and Corollary 4.8).

\section{Subspace Reduction \& Model Operator  }
\subsection{Subspace Reduction}
Due to the adiabatic nature of the problem let's consider the following subspace of  $H_0^1(\Omega_\epsilon)$, $$\mathcal{L}_\epsilon=\{u(x,y)=\chi(x)\sqrt{\frac{2}{\epsilon h(x)}}\sin(\frac{\pi y}{\epsilon h(x)}): \chi(x)\in H_0^1([-l_1, l_2])\},$$ where $H_0^1(\Omega_\epsilon)$ is the usual Sobolev Space which is also the natural domain of our Dirichlet Laplacian and $H_0^1([-l_1, l_2])$  is also the Sobolev Space as usual.

It is clear that $\mathcal{L}_\epsilon$ is a closed linear subspace of $H_0^1(\Omega_\epsilon)$.  Let $P$ be the orthogonal projection onto $\mathcal{L}_\epsilon.$ We also let $Q$ be the orthogonal projection onto the complement of $\mathcal{L}_\epsilon.$ Then clearly $P+Q=I.$

With these projections $P$ and $Q$, we have a decomposition of our Dirichlet Laplacian as follows:
 $$\Delta_\epsilon=A_{11}+A_{12}+A_{21}+A_{22}$$
where $A_{11}=P\Delta_\epsilon P, A_{12}=P\Delta_\epsilon Q, A_{21}=Q\Delta_\epsilon P$ and $A_{22}=Q\Delta_\epsilon Q.$ 
This decomposition allows us to rephrase our original eigenvalue  problem as an equivalent one shown by the following lemma.
\begin{lemma}
The Dirichlet eigenvalue problem 
\begin{align*}
\Delta_\epsilon u &=\Lambda_\epsilon u
\end{align*}
is equivalent to 
\begin{align*}
A_{11}u_1+A_{12}u_2 &=\Lambda_\epsilon u_1 \\
A_{21}u_1+A_{22}u_2 &=\Lambda_\epsilon u_2\\
\end{align*}
where $u_1=Pu, u_2=Qu,A_{11}=P\Delta_\epsilon P, A_{12}=P\Delta_\epsilon Q, A_{21}=Q\Delta_\epsilon P, A_{22}=Q\Delta_\epsilon Q,$ with $P$ the othorgonal projection onto \begin{align*}
\mathcal{L}_\epsilon=\{u(x,y)=\chi(x)\sqrt{\frac{2}{\epsilon h(x)}}\sin(\frac{\pi y}{\epsilon h(x)}): \chi(x)\in H_0^1([-l_1, l_2])\}
\end{align*}
and $Q=\mathbf{I}-P.$
\end{lemma}
\begin{proof}
Follows directly from definition.
\end{proof}

\subsection{Model Operator}
In this section, we will give an explicit formula for $A_{11}$ with which we are going  to study the eigenvalue asymptotics for $A_{11}$.

Recall that $A_{11}=P\Delta_\epsilon P$ and $P$ is orthogonal projection onto the closed subspace $\mathcal{L}_\epsilon.$

\begin{theorem}[Explicit Formula of $A_{11}$]
\begin{align*}
A_{11}=-\frac{d^2}{dx^2}+\frac{\pi^2}{\epsilon^2h^2}+(\frac{\pi^2}{3}+\frac{1}{4})\frac{h'^2}{h^2}
\end{align*}
\end{theorem}
\begin{proof}
Notice the energy form associated with $A_{11}$ is 
\begin{align*}
\mathcal{E}(u)&=\langle A_{11}u, u\rangle=\langle P\Delta_\epsilon Pu,u\rangle=\langle \Delta_\epsilon Pu, Pu\rangle\\
&=\int_{\Omega_\epsilon} |\nabla Pu|^2 dxdy
\end{align*}
where $\langle \cdot,\cdot\rangle$ represents the $L^2$ inner product.
Now assume $Pu=\chi(x)\sqrt{\frac{2}{\epsilon h(x)}}\sin(\frac{\pi y}{\epsilon h(x)})$, then 
\begin{tiny}
\begin{align*}
\mathcal{E}(u)&=\int_{\Omega_\epsilon}\left(\frac{d\chi(x)}{dx}\sqrt{\frac{2}{\epsilon h(x)}}\sin(\frac{\pi y}{\epsilon h(x)})+\chi(x)\frac{d\sqrt{\frac{2}{\epsilon h(x)}}\sin(\frac{\pi y}{\epsilon h(x)})}{dx}\right)^2+\left(\frac{\pi\chi(x)}{\epsilon h(x)}\sqrt{\frac{2}{\epsilon h(x)}}\sin(\frac{\pi y}{\epsilon h(x)})\right)^2 dxdy\\
&=\int_\mathbf{I} |\chi '|^2 dx+\int_\mathbf{I}[\frac{\pi^2}{\epsilon^2 h^2}+(\frac{\pi^2}{3}+\frac{1}{4})\frac{h'^2}{h^2}]\chi^2 dx
\end{align*}
\end{tiny}
where $\mathbf{I}=[-l_1,l_2]$. Clearly the associated differential operator is 
\begin{align*}
A_{11}=-\frac{d^2}{dx^2}+\frac{\pi^2}{\epsilon^2h^2}+(\frac{\pi^2}{3}+\frac{1}{4})\frac{h'^2}{h^2}
\end{align*}
defined on $H_0^1(\mathbf{I})$ with Dirichlet Boundary condition.


\end{proof}

It's known that spectrum of $A_{11}$ consists only eigenvalues. Clearly they depend on $\epsilon$. Now we will look at the dependence on $\epsilon$. More precisely we will find the full aymptotics of the eigenvalues.

It follows from Theorem 3.2 that $\mathcal{E}(u)\geq \frac{\pi^2}{\epsilon^2M^2}||u||_{\mathbf{L}^2(\mathcal{L}_\epsilon)}$, which gives the bottom of the spectrum. It is convenient to subtract the bottom of the spectrum from the engery form to get the following associated differential operator
\begin{align}
A_\epsilon=-\frac{d^2}{dx^2}+\frac{\pi^2}{\epsilon^2h^2}-\frac{\pi^2}{\epsilon^2M^2}+(\frac{\pi^2}{3}+\frac{1}{4})\frac{h'^2}{h^2}
\end{align}
on $H_0^1(\mathbf{I})$ with Dirichlet Boundary condition. Clearly $A_\epsilon=A_{11}-\frac{\pi^2}{\epsilon^2M^2}.$

\begin{lemma}
Let $\sigma(A_{11})$ be the spectrum of $A_{11}$, let $\sigma(A_\epsilon)$ be the spectrum of $A_\epsilon$. Then $\sigma(A_\epsilon)=\sigma(A_{11})-\frac{\pi^2}{\epsilon^2M^2}.$
\end{lemma}

\begin{proof}
The proof follows directly by noticing the spectrum of both operators is discrete, consisting only eigenvalues.
\end{proof}

Now we are going to study the eigenvalues of our modle operator $A_\epsilon$ in details. To show the ideas clearly and hence avoiding techinical complications, we will consider in subsection $3.3$ the case $h(x)=M-c_0x^m$ where $c_0$ is a constant and defer the general case $h(x)=M-c(x)x^m$ to subsection $3.4$, however the proof as we shall see goes parallely.

\subsection{Case 1: $h(x)=M-c_0x^m$}
The main idea involved in finding the full asymptotics for the eigenvalues of $A_\epsilon$ is to introduce a proper scaling. This is stated in the following Lemma.

\begin{lemma}
Let $x=\epsilon^{\alpha_1}y$ where $\alpha_1=\frac{2}{m+2},y\in \mathbf{I}_\epsilon=[-\frac{l_1}{\epsilon^{\alpha_1}},\frac{l_2}{\epsilon^{\alpha_1}}]$, then  
\begin{center}
$A_\epsilon$ is unitarily equivalent to $H_\epsilon= H_0+\sum_{n=1}^\infty H_n\epsilon^{n\alpha}$

\end{center}
 with $\alpha=m\alpha_1=\frac{2m}{m+2}, a_0=\frac{\pi^2}{M^2}, a_1=\frac{c_0}{M}, a=(\frac{\pi^2}{3}+\frac{1}{4})\frac{m^2c_0^2}{\pi^2}$ and 
$H_0=-\frac{d^2}{dy^2}+2a_0a_1y^m, H_n=(n+2)a_0a_1^{n+1}y^{nm+m}+(n-1)aa_0a_1^{n-2}y^{nm-2}$ are polynomial in $y$ of degree $nm+m.$
\end{lemma}
\begin{proof}
Notice  that $h(x)=M-cx^m$ and $\frac{\pi^2}{\epsilon^2 h^2}=\frac{\pi^2}{M^2\epsilon^2}\left[\sum_{n=1}^\infty n \left(\frac{cx^m}{M}\right)^{n-1}\right]$, so we have 
\begin{tiny}
\begin{align*}
A_\epsilon &=-\frac{d^2}{dx^2}+\frac{\pi^2}{\epsilon^2h^2}-\frac{\pi^2}{\epsilon^2M^2}+(\frac{\pi^2}{3}+\frac{1}{4})\frac{h'^2}{h^2}\\
&=-\frac{d^2}{dx^2}+\frac{\pi^2}{M^2\epsilon^2}\left[\sum_{n=2}^\infty n \left(\frac{cx^m}{M}\right)^{n-1}\right]+(\frac{\pi^2}{3}+\frac{1}{4})\frac{m^2c^2}{M^2}x^{2(m-1)}\left[\sum_{n=1}^\infty n \left(\frac{cx^m}{M}\right)^{n-1}\right]
\end{align*}
\end{tiny}

Buy introducing $x=\epsilon^{\alpha_1}y$ where $\alpha_1=\frac{2}{m+2}, y\in \mathbf{I}_\epsilon=[-\frac{l_1}{\epsilon^{\alpha_1}},\frac{l_2}{\epsilon^{\alpha_1}}]$, we see that 
\begin{tiny}
\begin{align*}
A_\epsilon=\frac{1}{\epsilon^{2\alpha_1}}\left(-\frac{d^2}{dy^2}+\frac{\pi^2}{M^2}\epsilon^{2\alpha_1-2}\left[\sum_{n=2}^\infty n \left(\frac{c(\epsilon^{\alpha_1-1}y)^m}{M}\right)^{n-1}\right]+(\frac{\pi^2}{3}+\frac{1}{4})\frac{m^2c^2}{M^2}(\epsilon^{\alpha_1}y)^{2m}y^{-2}\left[\sum_{n=1}^\infty n \left(\frac{c(\epsilon^{\alpha_1}y)^m}{M}\right)^{n-1}\right]\right)
\end{align*}
\end{tiny}

Now let $\alpha=m\alpha_1=\frac{2m}{m+2}, a_0=\frac{\pi^2}{M^2}, a_1=\frac{c_0}{M}, a=(\frac{\pi^2}{3}+\frac{1}{4})\frac{m^2c_0^2}{\pi^2}$. Then 
\begin{tiny}
\begin{align*}
A_\epsilon&=\frac{1}{\epsilon^{2\alpha_1}}\left(-\frac{d^2}{dy^2}+\frac{\pi^2}{M^2}\epsilon^{2\alpha_1-2}\left[\sum_{n=2}^\infty n \left(\frac{c(\epsilon^{\alpha_1}y)^m}{M}\right)^{n-1}\right]+(\frac{\pi^2}{3}+\frac{1}{4})\frac{m^2c^2}{M^2}(\epsilon^{\alpha_1}y)^{2m}y^{-2}\left[\sum_{n=1}^\infty n \left(\frac{c(\epsilon^{\alpha_1}y)^m}{M}\right)^{n-1}\right]\right)\\
&=\frac{1}{\epsilon^{2\alpha_1}}\left(-\frac{d^2}{dy^2}+2a_0a_1y^m+\sum_{n=1}^\infty\left[(n+2)a_0a_1^{n+1}y^{nm+m}+(n-1)aa_0a_1^{n-2}y^{nm-2}\right]\epsilon^{n\alpha}\right)
\end{align*}
\end{tiny}
\end{proof}

The proof of Lemma 3.4 above also shows the following.
\begin{lemma}
Let $\sigma(H_\epsilon)$ be the spectrum of $H_\epsilon$, then $\sigma(A_\epsilon)=\frac{1}{\epsilon^{2\alpha_1}}\sigma(H_\epsilon).$
\end{lemma}

\begin{proof}
The proof follows from the fact that the spectrum of both operators are discrete and also the computation in proving Lemma 3.4.
\end{proof}

Hence to figure out the asymptotics for $A_\epsilon$ we need to understand the asymptotics of eigenvalues of $H_\epsilon$. To study $H_\epsilon$ there are two major observations. The first is that as $\epsilon\rightarrow 0$, $\mathbf{I}_\epsilon\rightarrow \R.$ The second is that $H_0=-\frac{d^2}{dy^2}+2a_0a_1y^m$ initially defined on $\mathbf{I}_\epsilon$ can actually be viewed as the anharmonic oscillator $\tilde{H}_0=-\frac{d^2}{dy^2}+2a_0a_1y^m$ restricted to $\mathbf{I}_\epsilon.$ With these two oberservations one might expect perturbation theory around $\tilde{H}_0$ might give us a satisfactory result on studing the eignvalue asymptotics of $H_\epsilon.$ For further discussion, let $\tilde{H}_n=(n+2)a_0a_1^{n+1}y^{nm+m}+(n-1)aa_0a_1^{n-2}y^{nm-2}$ which is  the same as $H_n$ except that $\tilde{H}_n$ is defined on $\R.$ We also let $H_{\epsilon,K}=\tilde{H}_0+\sum_{n=1}^K \tilde{H}_n\epsilon^{n\alpha}.$

\begin{lemma}
Let $\{\mu_j\}_{j=0}^\infty$ be the full set of eigenvalues of $\tilde{H}_0$ defined on $H_0^1(\R)$ with corresponding eigenfunctions $\{\psi_j\}_{j=0}^\infty.$ 
Let $\lambda(H_{\epsilon,K})$ be the eigenvalue for $H_{\epsilon,K}=\tilde{H}_0+\sum_{n=1}^K \tilde{H}_n\epsilon^{n\alpha}$ defined on $H_0^1(\R)$ near $\mu_j$ with corresponding normalized eigenvector $\phi_{\epsilon,K}$. Then

\begin{align}
\lambda(H_{\epsilon,K}) 
 &\sim\mu_j+\sum_{n=1}^\infty \epsilon^{n\alpha}\tilde{q}_n
\end{align}
where 
\begin{tiny}
\begin{align*}
\tilde{q}_n=\sum_{k=1}^n\frac{1}{2\pi\mathrm{i}} \mathbf{Tr}\int_{\Gamma}\left(\sum_{j_1+j_2+\cdots+j_k=n, j_i\leq K, j_i\in \Z^+}(-1)^k\lambda(H_0-\lambda)^{-1}H_{j_1}(H_0-\lambda)^{-1}\cdots H_{j_k}(H_0-\lambda)^{-1} \right)d\lambda
\end{align*}
\end{tiny}
and $\Gamma=\{\lambda: |\lambda-\mu_j|= \delta\}$ any closed curve enclosing $\mu_j$ and inside which $H_{\epsilon,K}$ has single eigenvalue.
\end{lemma}

\begin{proof}
Regular Perturbation Theory. See Appendix. 
\end{proof}

We  also  show that the eigenfuctions $\phi_{\epsilon,K}$ of $H_{\epsilon,K}$ is decaying exponentially fast in the next Lemma.

\begin{lemma}
Let $V$ be a positive $\mathbf{C}^\infty$ function on $\R$ and let $H=-\Delta+V$. Suppose that $\psi$ is an eigenfunction of $H$. Then if $V(x)\geq s|x|^2-t$ for some $s$ and $t$, then for every $\epsilon>0$, there is a $D$ such that for all $x$
we have $$|\psi(x)|\leq De^{-\frac{1}{2}\sqrt{s-\epsilon}|x|^2}$$
\end{lemma}

\begin{proof}
 Reed-Simon Volumn 4, Page 252.
\end{proof}

\begin{corollary}
Let  $H_{\epsilon,K}=\tilde{H}_0+\sum_{n=1}^K\tilde{H}_n\epsilon^{n\alpha}$ defined on $\R$ where  $$\tilde{H}_0=-\frac{d^2}{dy^2}+2a_0a_1y^m, m\in 2\Z_+, \tilde{H}_n=(n+2)a_0a_1^{n+1}y^{nm+m}+(n-1)aa_0a_1^{n-2}y^{nm-2}$$
Let $\lambda(H_{\epsilon,K})$ be the eigenvalue for $H_{\epsilon,K}$ with corresponding normalized eigenvector $\phi_{\epsilon,K}$.Then there exists a $D$ such that for all $x$ we have 
$$|\phi_{\epsilon,K}(x)|\leq De^{-\frac{1}{2}\sqrt{a_0a_1}|x|^2}$$
\end{corollary}

\begin{proof}
Direct application of Lemma 3.7 with $s=2a_0a_1.$
\end{proof} 

With the eigenfunction $\phi_{\epsilon,K}$ we construct the following test function $\phi_K$ that will be used in proving our main result. The construction is stated in the following Lemma.

\begin{lemma}
Let $\phi_K=\phi_{\epsilon,K}\cdot f_\delta$, where $f_\delta(x)=\begin{cases} 1 &\mbox{if } x \in [-\frac{l_1}{\epsilon^{\alpha_1}}+\delta,\frac{l_2}{\epsilon^{\alpha_1}}-\delta] \\ 
0 & \mbox{if } x \notin  [-\frac{l_1}{\epsilon^{\alpha_1}},\frac{l_2}{\epsilon^{\alpha_1}}] \end{cases} $ and $f_\delta(x)\in \mathbf{C}^\infty(\R),1\geq  f_\delta(x)\geq 0.$ Then 
$$1\geq ||\phi_K||_{L^2(\mathbf{I}_\epsilon)}^2\geq 1-\frac{D^2}{\sqrt{a_0a_1}}e^{\delta\sqrt{a_0a_1}}(e^{-\sqrt{a_0a_1}\frac{l_1}{\epsilon^{\alpha_1}}}+e^{-\sqrt{a_0a_1}\frac{l_2}{\epsilon^{\alpha_1}}})$$
\end{lemma}
\begin{proof}
\begin{align*}
||\phi_K||_{L^2(\mathbf{I}_\epsilon)}^2 &= \int_{\mathbf{I}_\epsilon}f_\delta(x)^2\phi_{\epsilon,K}(x)^2 dx\\
&\geq \int_{\mathbf{I}_{\epsilon,\delta}}f_\delta(x)^2\phi_{\epsilon,K}(x)^2 dx\\
&=\int_{\mathbf{I}_{\epsilon,\delta}}\phi_{\epsilon,K}(x)^2 dx\\
&=\int_\R\phi_{\epsilon,K}(x)^2 dx-\int_{\R\backslash \mathbf{I}_{\epsilon,\delta}}\phi_{\epsilon,K}(x)^2 dx\\
&=1-\int_{\R\backslash \mathbf{I}_{\epsilon,\delta}}\phi_{\epsilon,K}(x)^2 dx\\
&\geq 1-\int_{\R\backslash \mathbf{I}_{\epsilon,\delta}}   [De^{-\frac{1}{2}\sqrt{a_0a_1}|x|^2}]^2    dx\\
&= 1-\int_{-\infty}^{-\frac{l_1}{\epsilon^{\alpha_1}}+\delta} D^2e^{-\sqrt{a_0a_1}|x|^2} dx-\int_{\frac{l_2}{\epsilon^{\alpha_1}}}^\infty D^2e^{-\sqrt{a_0a_1}|x|^2}dx\\
&\geq 1-\int_{-\infty}^{-\frac{l_1}{\epsilon^{\alpha_1}}+\delta} D^2e^{-\sqrt{a_0a_1}|x|} dx-\int_{\frac{l_2}{\epsilon^{\alpha_1}}}^\infty D^2e^{-\sqrt{a_0a_1}|x|}dx\\ 
&=1-\frac{D^2}{\sqrt{a_0a_1}}e^{\delta\sqrt{a_0a_1}}(e^{-\sqrt{a_0a_1}\frac{l_1}{\epsilon^{\alpha_1}}}+e^{-\sqrt{a_0a_1}\frac{l_2}{\epsilon^{\alpha_1}}})
\end{align*}
where $\mathbf{I}_{\epsilon,\delta}=[-\frac{l_1}{\epsilon^{\alpha_1}}+\delta,\frac{l_2}{\epsilon^{\alpha_1}}-\delta] $

On the other hand 
\begin{align*}
||\phi_K||_{L^2(\mathbf{I}_\epsilon)}^2 &= \int_{\mathbf{I}_\epsilon}f_\delta(x)^2\phi_{\epsilon,K}(x)^2 dx\\
&\leq \int_{\mathbf{I}_\epsilon}\phi_{\epsilon,K}(x)^2 dx\\
&\leq \int_\R\phi_{\epsilon,K}(x)^2 dx\\
&\leq 1
\end{align*}
\end{proof}

Now we can state the main results in the following theorem about the full eigenvalue asymptotics of $H_\epsilon=H_0+\sum_{n=1}^\infty H_n\epsilon^{n\alpha}$ defined over $H_0^1(\mathbf{I}_\epsilon).$

\begin{theorem}
Let $\{\mu_j\}_{j=0}^\infty$ be the full set of eigenvalues of $\tilde{H}_0=-\frac{d^2}{dy^2}+2a_0a_1y^m$ defined on $H_0^1(\R)$ with corresponding eigenfunctions $\{\psi_j\}_{j=0}^\infty.$ Then the perturbed eigenvalue $\nu$ for $H_\epsilon$ around $\mu_j$ has asymptotic expansion given by 
\begin{align}
\nu&\sim\mu_j+\sum_{n=1}^\infty q_n\epsilon^{n\alpha}
\end{align}
where 
\begingroup\makeatletter\def\f@size{9}\check@mathfonts
\def\maketag@@@#1{\hbox{\m@th\large\normalfont#1}}%
\begin{align*}
q_n
 &=\sum_{k=1}^n\frac{1}{2\pi\mathrm{i}} \mathbf{Tr}\int_{\Gamma}\left(\sum_{j_1+j_2+\cdots+j_k=n, j_i\in \Z^{+}}(-1)^k\lambda(H_0-\lambda)^{-1}H_{j_1}(H_0-\lambda)^{-1} H_{j_2}(H_0-\lambda)^{-1} \cdots H_{j_k}(H_0-\lambda)^{-1} \right)d\lambda\\
&= \sum_{k=1}^n\;\; \sum_{j_1+j_2+\cdots+j_k=n, j_i\in \Z^{+}}\;\; \sum_{s_1,s_2,\cdots,s_k=1}^\infty \frac{1}{2\pi\mathfrak{i}}\int_{\Gamma}\frac{(-1)^k\lambda}{\mu_{s_1}-\lambda}\cdot\frac{a_{j_ks_{1}s_{2}}}{\mu_{s_1}-\lambda}\cdot \frac{a_{j_{k-1}s_{2}s_{3}}}{\mu_{s_2}-\lambda}\cdots\cdot\frac{a_{j_1s_ks_1}}{\mu_{s_k}-
\lambda}d\lambda
\end{align*}\endgroup
with $\Gamma=\{\lambda: |\lambda-\mu_j|= \delta\}$ any closed curve enclosing $\mu_0$ and inside which $H_\epsilon$ has single eigenvalue and $a_{nsk}=\langle H_n\psi_s, \psi_k\rangle.$
\end{theorem}
\begin{corollary}
\begin{align*}
q_1&=\frac{1}{2\pi\mathfrak{i}} \mathbf{Tr}\int_{\Gamma=\{|\lambda-\lambda_0|\leq \epsilon\}}(-1)\lambda(H_0-\lambda)^{-1}H_1(H_0-\lambda)^{-1} d\lambda=-a_{100}\\
q_2&=-a_{200}+\sum_{j=1}^\infty a_{10j}a_{1j0}\frac{\mu_j}{(\mu_j-\mu_0)^2}-\sum_{s=1}^\infty a_{1s0}a_{10s}\frac{\mu_0}{(\mu_s-\mu_0)^2}=\sum_{s=1}^\infty a_{1s0}a_{10s}\frac{1}{\mu_s-\mu_0}-a_{200}
\end{align*}
\end{corollary}

Before proving this Theorem 3.10. The following observation is important.

\begin{lemma}
\begin{align*}
\lambda(H_{\epsilon,K})^{(K)}&\widehat{=}\mu_j+\sum_{n=1}^K \tilde{q}_n\epsilon^{n\alpha}\\
&=\mu_j+\sum_{n=1}^K q_n\epsilon^{n\alpha}
\end{align*}
\end{lemma}
\begin{proof}
\begingroup\makeatletter\def\f@size{9}\check@mathfonts
\def\maketag@@@#1{\hbox{\m@th\large\normalfont#1}}%
\begin{align*}
\lambda(H_{\epsilon,K})^{(K)}&=\mu_j+\sum_{n=1}^K \tilde{q}_n\epsilon^{n\alpha}\\
&=\mu_j+\sum_{n=1}^K \epsilon^{n\alpha}\sum_{k=1}^n\frac{1}{2\pi\mathfrak{i}} \mathbf{Tr}\int_{\Gamma}\left(\sum_{j_1+j_2+\cdots+j_k=n, j_i\leq K, j_i\in \Z^+}\frac{(-1)^k\lambda}{H_0-\lambda}H_{j_1}(H_0-\lambda)^{-1} \cdots H_{j_k}(H_0-\lambda)^{-1} \right)d\lambda  \\
&=\mu_j+\sum_{n=1}^K q_n\epsilon^{n\alpha}
\end{align*}\endgroup
\end{proof}

\begin{proof}
\textbf{(Proof of Theorem 3.10)} The main idea involved in proving the Theorem is to show that for all $K,$ 
\begin{align}
||H_\epsilon-\left(\mu_j+\sum_{n=1}^K q_n\epsilon^{n\alpha}\right)||=O(\epsilon^{(K+1)\alpha})
\end{align}
Then by the self adjointness of the $H_\epsilon$, we have \begin{align*}
\nu&\sim\mu_j+\sum_{n=1}^\infty q_n\epsilon^{n\alpha}.
\end{align*}

Now we will prove $(3.4)$ as follows. Using Lemma 3.12
\begin{small}
\begin{align*}
&||H_\epsilon\phi_K-\left(\mu_j+\sum_{n=1}^K q_n\epsilon^{n\alpha}\right)\phi_K||\\
=&
||H_\epsilon\phi_K-\lambda(H_{\epsilon,K})^{(K)}\phi_K||\\
 =&|[H_{\epsilon,K}+(H_\epsilon-H_{\epsilon,K})]\phi_K-\lambda(H_{\epsilon,K})^{(K)}\phi_K| \\
 =&||(H_\epsilon-H_{\epsilon,K})\phi_K+[H_{\epsilon,K}-\lambda(H_{\epsilon,K})^{(K)}]\phi_K||\\
=&||(H_\epsilon-H_{\epsilon,K})\phi_K+[\lambda_{H_{\epsilon,K}}-\lambda(H_{\epsilon,K})^{(K)}]\phi_K-\phi_{\epsilon,K}f_\delta''-2\phi_{\epsilon,K}'f_\delta'||\\
\leq& O(\epsilon^{(K+1)\alpha})||\phi_K||_{L^2(\mathbf{I}_\epsilon)}+||(H_\epsilon-H_{\epsilon,K})\phi_K||+||\phi_{\epsilon,K}f_\delta''||+||2\phi_{\epsilon,K}'f_\delta'||
\end{align*}
\end{small}

Notice now that $||\phi_{\epsilon,K}f_\delta''||+||2\phi_{\epsilon,K}'f_\delta'||\rightarrow 0$ when $\delta\rightarrow 0$ by absolute continuity and the fact  that $f_\delta$ is supported on $[-\frac{l_1}{\epsilon^\alpha_1},-\frac{l_1}{\epsilon^\alpha_1}+\delta]\cup [\frac{l_2}{\epsilon^\alpha_1}-\delta,\frac{l_2}{\epsilon^\alpha_1}].$ So to prove the theorem, it suffices to show $||(H_\epsilon-H_{\epsilon,K})\phi_K||=O(\epsilon^{(K+1)\alpha}).$

In fact,  because $H_\epsilon$ is defined on   $\mathbf{I}_\epsilon=[-\frac{l_1}{\epsilon^{\alpha_1}},\frac{l_2}{\epsilon^{\alpha_1}}]$, we have that $$a_1\epsilon^\alpha y^m\leq \max\{c\frac{l_1^m}{M},c\frac{l_2^m}{M}\}\hat{=}q<1 \quad \text{since}\quad h(x)>0 $$

And notice 
\begin{small}
\begin{align*}
&\sum_{n=K+1}^\infty H_n\epsilon^{n\alpha}\\
=& \sum_{n=K+1}^\infty [a_0\frac{1}{\epsilon^\alpha}(n+2)(a_1y^m\epsilon^\alpha)^{n+1}+aa_0\epsilon^{2\alpha}y^{2m-2}(n-1)(a_1y^m\epsilon^\alpha)^{n-2}]\\
=&a_0\frac{1}{\epsilon^\alpha}\frac{(K+3)(a_1y^m\epsilon^\alpha)^{K+2}-(K+2)(a_1y^m\epsilon^\alpha)^{K+3}}{(1-a_1y^m\epsilon^\alpha)^2}+aa_0\epsilon^{2\alpha}y^{2m-2}\frac{K(a_1y^m\epsilon^\alpha)^{K-1}+(1-K)(a_1y^m\epsilon^\alpha)^K}{(1-a_1y^m\epsilon^\alpha)^2}\\
\leq & a_0\frac{1}{(1-q)^2}[(K+3)(a_1y^m)^{K+2}\epsilon^{(K+1)\alpha}-(K+2)(a_1y^m)^{K+3}\epsilon^{(K+2)\alpha}]\\
&+aa_0\frac{1}{(1-q)^2}[K(a_1y^m)^{K-1}y^{2m-2}\epsilon^{(K+1)\alpha}-(K-1)y^{2m-2}(a_1y^m)^K\epsilon^{(K+2)\alpha}]\\
\leq & \epsilon^{(K+1)\alpha}[\frac{a_0}{(1-q)^2}(K+3)(a_1y^m)^{K+2}+\frac{aa_0}{(1-q)^2}Ka_1^{K-1}y^{(K+1)m-2}]\\
\leq & \epsilon^{(K+1)\alpha}Cy^{(K+2)m}
\end{align*}
\end{small}
where $C$ is some constant depending only  on $K, a, a_0, a_1$ and $q$. We also notice from Corollary 3.8 that 
\begin{small}
\begin{align*}
||y^{(K+2)m}\phi_K||_{L^2(\mathbf{I}_\epsilon)}^2=\int_{\mathbf{I}_\epsilon} y^{2(K+2)m}\phi_K^2\leq \int_{\R}y^{2(K+2)m}\phi_{\epsilon,K}^2\leq\int_\R y^{2(K+2)m}D^2e^{-\sqrt{a_0a_1}|y|^2}<\infty
\end{align*}
\end{small}
Most importantly the bound here is not involving $\epsilon.$
Thus $$||(H_\epsilon-H_{\epsilon,K})\phi_K|=||\sum_{n=K+1}^\infty H_n\epsilon^{n\alpha}\phi_K||\leq ||\epsilon^{(K+1)\alpha}Cy^{(K+2)m}\phi_K||=O(\epsilon^{(K+1)\alpha})$$
In conclusion we just showed that 
$\nu\sim \mu_j+\sum_{n=1}^\infty q_n\epsilon^{n\alpha}$.

\end{proof}
\subsection{Case 2: $h(x)=M-c(x)x^m$}

The situation for the general case $h(x)=M-c(x)x^m$ is very  much similar to the previous case.  For later discussion let's recall $c(x)=\sum_{n=0}^\infty c_nx^n$ is analytic. Because of that we also have the following analytic functions
\begin{align*}
\label{expansions}
c(x)^2&=\sum_{n=0}^\infty d_nx^n &
c'(x)&=\sum_{n=0}^\infty e_nx^n \\
c'(x)^2&=\sum_{n=0}^\infty g_nx^n &
c(x)c'(x)&=\sum_{n=0}^\infty f_nx^n \\
\left(\frac{c(x)}{M}\right)^{n+1}&=\sum_{k=0}^\infty \alpha_{k,n}x^k &
\left(\frac{c(x)}{M}\right)^{n-2}&=\sum_{k=0}^\infty \beta_{k,n}x^k\\
\end{align*}
and with the convention that $t_j=0$ for all the negative indicies with $t_j$ being any of those coefficients in the expansions above. And for simplification of the notations, all the notations will be understood within the context of this section 3.4 and should not be viewed as conflicted with the same notations in other sections.

This general case is very much similar as the Case 1 in the sense that the results are parallel to the previous case. 
First by introducing a proper scaling we  transformed $A_\epsilon=A_{11}-\frac{\pi^2}{\epsilon^2M^2}$ unitarily to $H_\epsilon$ as below.
\begin{lemma}
Let $x=\epsilon^{\alpha_1}y$ where $\alpha_1=\frac{2}{m+2},y\in \mathbf{I}_\epsilon=[-\frac{l_1}{\epsilon^{\alpha_1}},\frac{l_2}{\epsilon^{\alpha_1}}]$, then
\begin{center}
$A_\epsilon$ is unitarily equivalent to $H_\epsilon= H_0+\sum_{n=1}^\infty H_n\epsilon^{n\alpha}$

\end{center}
with $a_0=\frac{\pi^2}{M^2}$ and  $H_0=-\frac{d^2}{dy^2}+2a_0a_1y^m,$

\begin{align*}
H_n=\frac{2a_0c_n}{M}y^{n+m}+\sum_{k+sm=n,k\in \mathbf{N},s\geq 1} \left[(s+2)a_0\alpha_{k,s}y^{n+m}+\sum_{i+j=k,i,j\in \mathbf{N}}\beta_{i,s}y^i\gamma_j\right]
\end{align*}
where $\gamma_j=A_1d_jy^j+A_2f_{j-1}y^{j-1}+A_3g_{j-2}y^{j-2}$ and $A_1=(\frac{\pi^2}{3}+\frac{1}{4})\frac{m^2}{\pi^2}(n-1)a_0, A_2=\frac{2m}{M^2}(\frac{\pi^2}{3}+\frac{1}{4})(n-1), A_3=(\frac{\pi^2}{3}+\frac{1}{4})(n-1)\frac{1}{M^2}$ are some pure constants.
\end{lemma}
\begin{proof}
The proof is similar as Lemma 3.4. And it is contained in Appendix.
\end{proof}

Secondly $H_0=-\frac{d^2}{dy^2}+2a_0a_1y^m$ initially defined on $\mathbf{I}_\epsilon$ can actually be viewed as the anharmonic oscillator $\tilde{H}_0=-\frac{d^2}{dy^2}+2a_0a_1y^m$ restricted to $\mathbf{I}_\epsilon.$ With this oberservation one might expect perturbation theory around $\tilde{H}_0$ might give us a satisfactory result on studing the eignvalue asymptotics of $H_\epsilon.$ For further discussion, let $\tilde{H}_n=H_n$ which is  the same polynomial in $y$ of degree $n+m$ as $H_n$ except that $\tilde{H}_n$ is defined on $\R.$ We also let $H_{\epsilon,K}=\tilde{H}_0+\sum_{n=1}^K \tilde{H}_n\epsilon^{n\alpha_1}.$

\begin{lemma} Let $\{\mu_j\}_{j=0}^\infty$ be the full set of eigenvalues of $\tilde{H}_0$ defined on $H_0^1(\R)$ with corresponding eigenfunctions $\{\psi_j\}_{j=0}^\infty.$
Let $\lambda(H_{\epsilon,K})$ be the eigenvalue for $H_{\epsilon,K}=\tilde{H}_0+\sum_{n=1}^K \tilde{H}_n\epsilon^{n\alpha}$ defined on $H_0^1(\R)$ near $\mu_j$ with corresponding normalized eigenvector $\phi_{\epsilon,K}$. 
Then

\begin{align}
\lambda(H_{\epsilon,K}) 
 &\sim\mu_j+\sum_{n=1}^\infty \epsilon^{n\alpha_1}\tilde{q}_n
\end{align}
where 
\begin{small}
\begin{align*}
\tilde{q}_n=\sum_{k=1}^n\frac{1}{2\pi\mathrm{i}} \mathbf{Tr}\int_{\Gamma}\left(\sum_{j_1+j_2+\cdots+j_k=n, j_i\leq K, j_i\in \Z^+}(-1)^k\lambda(H_0-\lambda)^{-1}H_{j_1}(H_0-\lambda)^{-1}\cdots H_{j_k}(H_0-\lambda)^{-1} \right)d\lambda
\end{align*}
\end{small}
and $\Gamma=\{\lambda: |\lambda-\mu_j|= \delta\}$ any closed curve enclosing $\mu_j$ and inside which $H_{\epsilon,K}$ has single eigenvalue.
\end{lemma}
\begin{proof}
The proof is similar as Lemma 3.6 by doing Regular Perturbation Theory. See Appendix.
\end{proof}

We  also  show that the eigenfuctions $\phi_{\epsilon,K}$ of $H_{\epsilon,K}$ is decaying exponentially fast in the next Lemma.

\begin{lemma}
Let  $H_{\epsilon,K}=H_0+\sum_{n=1}^KH_n\epsilon^{n\alpha_1}$ defined on $\R$ where  $$H_0=-\frac{d^2}{dy^2}+\frac{2a_0c_0}{M}y^m,  H_n=\frac{2a_0c_n}{M}y^{n+m}+\sum_{k+sm=n,k\in \mathbf{N},s\geq 1} \left[(s+2)a_0\alpha_{k,s}y^{n+m}+\sum_{i+j=k,i,j\in \mathbf{N}}\beta_{i,s}y^i\gamma_j\right]$$
Let $\lambda(H_{\epsilon,K})$ be the eigenvalue for $H_{\epsilon,K}$ with corresponding normalized eigenvector $\phi_{\epsilon,K}$.Then there exists a $D$ such that for all $x$ we have 
$$|\phi_{\epsilon,K}(x)|\leq De^{-\frac{1}{2}\sqrt{a_0a_1}|x|^2}$$ where $a_1=\frac{c_0}{M}.$
\end{lemma}

\begin{proof}
Direct application of Lemma 3.7 with $c=2a_0a_1$ where $a_1=\frac{c_0}{M}.$ 
\end{proof}

With the eigenfunction $\phi_{\epsilon,K}$ we construct the following test function $\phi_K$ that will be used in proving our main result as below.

\begin{lemma}
Let $\phi_K=\phi_{\epsilon,K}\cdot f_\delta$, where $f_\delta(x)=\begin{cases} 1 &\mbox{if } x \in [-\frac{l_1}{\epsilon^{\alpha_1}}+\delta,\frac{l_2}{\epsilon^{\alpha_1}}-\delta] \\ 
0 & \mbox{if } x \notin  [-\frac{l_1}{\epsilon^{\alpha_1}},\frac{l_2}{\epsilon^{\alpha_1}}] \end{cases} $ and $f_\delta(x)\in \mathbf{C}^\infty(\R),1\geq  f_\delta(x)\geq 0.$ Then 
$$1\geq ||\phi_K||_{L^2(\mathbf{I}_\epsilon)}^2\geq 1-\frac{D^2}{\sqrt{a_0a_1}}e^{\delta\sqrt{a_0a_1}}(e^{-\sqrt{a_0a_1}\frac{l_1}{\epsilon^{\alpha_1}}}+e^{-\sqrt{a_0a_1}\frac{l_2}{\epsilon^{\alpha_1}}})$$
\end{lemma}
\begin{proof}
The proof is the same as Lemma 3.9.
\end{proof}

Now we are ready to state the main results about the full eigenvalue asymptotics of $H_\epsilon=H_0+\sum_{n=1}^\infty H_n\epsilon^{n\alpha}$ defined over $H_0^1(\mathbf{I}_\epsilon).$

\begin{theorem}
Let $\{\mu_j\}_{j=0}^\infty$ be the full set of eigenvalues of $\tilde{H}_0=-\frac{d^2}{dy^2}+2a_0a_1y^m$ defined on $H_0^1(\R)$  with corresponding eigenfunctions $\{\psi_j\}_{j=0}^\infty.$ Then the perturbed eigenvalue $\nu$ for $H_\epsilon$ around $\mu_j$ has asymptotic expansion given by 
\begin{align}
\nu&\sim\mu_j+\sum_{n=1}^\infty q_n\epsilon^{n\alpha_1}
\end{align}
where 
\begingroup\makeatletter\def\f@size{9}\check@mathfonts
\def\maketag@@@#1{\hbox{\m@th\large\normalfont#1}}%
\begin{align*}
q_n
 &=\sum_{k=1}^n\frac{1}{2\pi\mathrm{i}} \mathbf{Tr}\int_{\Gamma}\left(\sum_{j_1+j_2+\cdots+j_k=n, j_i\in \Z^{+}}(-1)^k\lambda(H_0-\lambda)^{-1}H_{j_1}(H_0-\lambda)^{-1} H_{j_2}(H_0-\lambda)^{-1} \cdots H_{j_k}(H_0-\lambda)^{-1} \right)d\lambda\\
&= \sum_{k=1}^n\;\; \sum_{j_1+j_2+\cdots+j_k=n, j_i\in \Z^{+}}\;\; \sum_{s_1,s_2,\cdots,s_k=1}^\infty \frac{1}{2\pi\mathfrak{i}}\int_{\Gamma}\frac{(-1)^k\lambda}{\mu_{s_1}-\lambda}\cdot\frac{a_{j_ks_{1}s_{2}}}{\mu_{s_1}-\lambda}\cdot \frac{a_{j_{k-1}s_{2}s_{3}}}{\mu_{s_2}-\lambda}\cdots\cdot\frac{a_{j_1s_ks_1}}{\mu_{s_k}-
\lambda}d\lambda
\end{align*}\endgroup
with $\Gamma=\{\lambda: |\lambda-\mu_j|= \delta\}$ any closed curve enclosing $\mu_j$ and inside which $H_\epsilon$ has single eigenvalue and $a_{nsk}=\langle H_n\psi_s, \psi_k\rangle.$
\end{theorem}

Before proving this Theorem 3.17. The following observation is important. 

\begin{lemma}
\begin{align*}
\lambda(H_{\epsilon,K})^{(K)}&\widehat{=}\mu_j+\sum_{n=1}^K \tilde{q}_n\epsilon^{n\alpha_1}\\
&=\mu_j+\sum_{n=1}^K q_n\epsilon^{n\alpha_1}
\end{align*}
\end{lemma}
\begin{proof}
\begingroup\makeatletter\def\f@size{9}\check@mathfonts
\def\maketag@@@#1{\hbox{\m@th\large\normalfont#1}}%
\begin{align*}
\lambda(H_{\epsilon,K})^{(K)}&=\mu_j+\sum_{n=1}^K \tilde{q}_n\epsilon^{n\alpha_1}\\
&=\mu_j+\sum_{n=1}^K \epsilon^{n\alpha_1}\sum_{k=1}^n\frac{1}{2\pi\mathfrak{i}} \mathbf{Tr}\int_{\Gamma}\left(\sum_{j_1+j_2+\cdots+j_k=n, j_i\leq K, j_i\in \Z^+}\frac{(-1)^k\lambda}{H_0-\lambda}H_{j_1}(H_0-\lambda)^{-1} \cdots H_{j_k}(H_0-\lambda)^{-1} \right)d\lambda  \\
&=\mu_j+\sum_{n=1}^K q_n\epsilon^{n\alpha_1}
\end{align*}\endgroup
\end{proof}

\begin{proof}
\textbf{(Proof of Theorem 3.17)} The main idea involved in proving the Theorem is to show that for all $K,$ 
\begin{align}
||H_\epsilon-\left(\mu_j+\sum_{n=1}^K q_n\epsilon^{n\alpha_1}\right)||=O(\epsilon^{(K+1)\alpha_1})
\end{align}
Then by the self adjointness of the $H_\epsilon$, we have \begin{align*}
\nu&\sim\mu_j+\sum_{n=1}^\infty q_n\epsilon^{n\alpha_1}.
\end{align*}

Now we will prove $(3.7)$ as follows. Using Lemma 3.18
\begin{small}
\begin{align*}
&||H_\epsilon\phi_K-\left(\mu_j+\sum_{n=1}^K q_n\epsilon^{n\alpha_1}\right)\phi_K||\\
=&
||H_\epsilon\phi_K-\lambda(H_{\epsilon,K})^{(K)}\phi_K||\\
 =&|[H_{\epsilon,K}+(H_\epsilon-H_{\epsilon,K})]\phi_K-\lambda(H_{\epsilon,K})^{(K)}\phi_K| \\
 =&||(H_\epsilon-H_{\epsilon,K})\phi_K+[H_{\epsilon,K}-\lambda(H_{\epsilon,K})^{(K)}]\phi_K||\\
=&||(H_\epsilon-H_{\epsilon,K})\phi_K+[\lambda_{H_{\epsilon,K}}-\lambda(H_{\epsilon,K})^{(K)}]\phi_K-\phi_{\epsilon,K}f_\delta''-2\phi_{\epsilon,K}'f_\delta'||\\
\leq& O(\epsilon^{(K+1)\alpha_1})||\phi_K||_{L^2(\mathbf{I}_\epsilon)}+||(H_\epsilon-H_{\epsilon,K})\phi_K||+||\phi_{\epsilon,K}f_\delta''||+||2\phi_{\epsilon,K}'f_\delta'||
\end{align*}
\end{small}

Notice now that $||\phi_{\epsilon,K}f_\delta''||+||2\phi_{\epsilon,K}'f_\delta'||\rightarrow 0$ when $\delta\rightarrow 0$ by absolute continuity and the fact  that $f_\delta$ is supported on $[-\frac{l_1}{\epsilon^\alpha_1},-\frac{l_1}{\epsilon^\alpha_1}+\delta]\cup [\frac{l_2}{\epsilon^\alpha_1}-\delta,\frac{l_2}{\epsilon^\alpha_1}].$ So to prove the theorem, it suffices to show $||(H_\epsilon-H_{\epsilon,K})\phi_K||=O(\epsilon^{(K+1)\alpha_1}).$

And notice from analyticity  in $\epsilon$ we have
\begin{align*}
&\sum_{n=K+1}^\infty H_n\epsilon^{n\alpha_1}\leq C\epsilon^{(K+1)\alpha_1}H_{K+1}\leq C\epsilon^{(K+1)\alpha_1}|y|^{K+1+m}\\
\end{align*} 
where $C$ is some constant not depending on $\epsilon.$ But we also notice from Lemma 3.16 that

$$||y^{(K+1+m)}\phi_K||_{L^2(\mathbf{I}_\epsilon)}^2=\int_{\mathbf{I}_\epsilon} y^{2(K+1+m)}\phi_K^2\leq \int_{\R}y^{2(K+1+m)}\phi_{\epsilon,K}^2\leq\int_\R y^{2(K+1+m)}D^2e^{-\sqrt{a_0a_1}|y|^2}<\infty$$
Most importantly the bound here is not involving $\epsilon.$

Thus $||(H_\epsilon-H_{\epsilon,K})\phi_K|=||\sum_{n=K+1}^\infty H_n\epsilon^{n\alpha_1}\phi_K||\leq ||\epsilon^{(K+1)\alpha}Cy^{(K+1+m)}\phi_K||=O(\epsilon^{(K+1)\alpha_1})$

In conclusion we just showed that 
$\nu\sim \mu_j+\sum_{n=1}^\infty q_n\epsilon^{n\alpha_1}$.

\end{proof}

In summary we have the following result about full asymptotics of our model operator $A_{11}$. 
\begin{theorem}
Let $\{\mu_j\}_{j=0}^\infty$ be the full set of eigenvalues of $\tilde{H}_0$ defined on $H_0^1(\R)$. Then  $\epsilon^{2\alpha_1}\left(A_{11}-\frac{\pi^2}{M^2\epsilon^2}\right)$ has full eigenvalue asymptotics  given by 
\begin{align*}
\nu&\sim\mu_j+\sum_{n=1}^\infty q_n\epsilon^{n\alpha_1}
\end{align*}
where $q_n$ can be computed explicitly.
\end{theorem}
\begin{proof}
Follows directly from Theorem 3.18 and Lemma 3.5.
\end{proof}

\section{Study of Difference $\tilde{\lambda}=\Lambda_\epsilon-\lambda$}

In the previous sections we studied in details about the asymptotics of the eigenvalue of $A_{11}$. More precisely we showed  $\epsilon^{2\alpha_1}[A_{11}-\frac{\pi^2}{M^2\epsilon^2}]$ has eigenvalue asymptotics $\nu\sim \mu_0+\sum_{n=1}^\infty q_n\epsilon^{n\alpha_1}$. To go back to the full asymptotics of Dirichelt Laplacian, we only need to figure out  the difference $\tilde{\lambda}=\Lambda-\lambda$ between the eigenvalues $\Lambda$ of the original operator and the eigenvalues $\lambda$ of  the model operator $A_{11}$. For later discussions let $\lambda$ be the eigenvalues of $A_{11}$ with corresponding normalized eigenfunction $\phi$. Then $A_{11}\phi=\lambda\phi.$

For the rest of the paper in section 4.1 we will derive the equation that is satisfied by $\tilde{\lambda}$ and in section 4.2 we will derive an iterative scheme for solving the equation. 
\subsection{{Equation of $\tilde{\lambda}$}}
To get the equation for $\tilde{\lambda}$, recall from Lemma 3.1, we have 
\begin{align}
\begin{split}
A_{11}u_1+A_{12}u_2 &=(\lambda+\tilde{\lambda})u_1 \\
A_{21}u_1+A_{22}u_2 &=(\lambda+\tilde{\lambda})u_2\\
\end{split}
\end{align}
where $u_1=Pu, u_2=Qu.$ The two functions $u_1$ and $ u_2$ are actually connected as shown in the following Lemma.

\begin{lemma}
\begin{align}
u_2
&=-\sum_{n=0}^\infty \tilde{\lambda}^n(A_{22}-\lambda)^{-(n+1)}A_{21}u_1
\end{align}
\end{lemma}

\begin{proof}
From $(4.1)$ we have 
\begin{align}
u_2&=(A_{22}-\lambda)^{-1}(\tilde{\lambda}u_2-A_{21}u_1)
\end{align}
Using $(4.3)$ iteratively, we have 
\begin{align*}
u_2&=-[I-\tilde{\lambda}(A_{22}-\lambda)^{-1}]^{-1}[(A_{22}-\lambda)^{-1}A_{21}u_1]
\end{align*}
Inside the proof there are two things that should be explained more, namely the invertibility of $(A_{22}-\lambda)$ and the invertibility of 
$I-\tilde{\lambda}(A_{22}-\lambda)^{-1}$.

$(A_{22}-\lambda)$ is invertible because of the following observations. First we already see that $\lambda\sim \frac{\pi^2}{M^2\epsilon^2}$. Using variational characterization (Energy Estimate) one easily see that the spectrum of $A_{22}=Q\Delta_\epsilon Q$ is bounded below by $\frac{4\pi^2}{M^2\epsilon^2}$.  More precisely one can show  $\frac{\langle A_{22}v, v\rangle	}{\langle v, v\rangle }\geq \frac{4\pi^2}{\epsilon^2M^2}, \forall v\in \mathrm{L}^2(\Omega_\epsilon).$
So $\lambda$ is away from the spectrum of $A_{22}$. And this particularly implies that $(A_{22}-\lambda)$ is invertible.  Moreover one have $$||(A_{22}-\lambda)^{-1}||\leq C\epsilon^2$$ for some pure constant $C$.

$I-\tilde{\lambda}(A_{22}-\lambda)^{-1}$ is invertible for the following reasons. In the work \cite{LS}, the authors proved
\begin{align}
\lim_{\epsilon\rightarrow 0} \epsilon^{2\alpha_1}\left(\Lambda_j(\epsilon)-\frac{\pi^2}{M^2\epsilon^2}\right)=\mu_j
\end{align}
And we also proved that in Theorem 3.19 that 
\begin{align}
\lim_{\epsilon\rightarrow 0}\epsilon^{2\alpha_1}\left(\lambda_j-\frac{\pi^2}{M^2\epsilon^2}\right)=\mu_j
\end{align}
Combine (4.4) and (4.5) we have
\begin{align*}
\lim_{\epsilon\rightarrow 0}\epsilon^{2\alpha_1}\left(\Lambda_j-\lambda_j\right)=0
\end{align*}
namely
\begin{align}
\lim_{\epsilon\rightarrow 0}\epsilon^{2\alpha_1}\tilde{\lambda}=0
\end{align}
 Recall $\alpha_1=\frac{2}{m+2}\leq \frac{1}{2}$ since $m$ is even integer. So the fact $||(A_{22}-\lambda)^{-1}||\leq C\epsilon^2$ combined with $(4.6)$ will guarantee that $||\tilde{\lambda}(A_{22}-\lambda)^{-1}|<1$ as $\epsilon\rightarrow 0.$ And this shows $I-\tilde{\lambda}(A_{22}-\lambda)^{-1}$ is invertible and one even have the Neumann Expansion.
 \begin{align*}
\left[ I-\tilde{\lambda}(A_{22}-\lambda)^{-1}\right]^{-1}=\sum_{n=0}^\infty \tilde{\lambda}^n(A_{22}-\lambda)^{-n}
 \end{align*}
\end{proof}

Now we can state the equation that is satisfied by $\tilde{\lambda}.$

\begin{lemma}[Equation for $\tilde{\lambda}$]
 Let $\phi$ be the normalized eigenfunctions of $A_{11}\phi=\lambda\phi$. Then
\begin{align}
-\tilde{\lambda}=\frac{\langle u_1, A_{12}(A_{22}-\lambda)^{-1}A_{21}\phi\rangle}{\langle u_1,\phi\rangle}+\sum_{n=1}^\infty\tilde{\lambda}^n\frac{\langle u_1, A_{12}(A_{22}-\lambda)^{-n-1}A_{21}\phi\rangle}{\langle u_1,\phi\rangle}
\end{align}
\end{lemma}

\begin{proof}
Notice
\begin{align*}
\tilde{\lambda}&=\frac{\langle (\Delta-\lambda)u, \phi\rangle}{\langle u, \phi\rangle}
=\frac{\langle (P\Delta P+P\Delta Q-\lambda)u, P\phi\rangle}{\langle u, \phi\rangle}
=\frac{\langle u_2,\Delta\phi\rangle}{\langle u_1, \phi\rangle}
=\frac{\langle A_{12}u_2, \phi\rangle}{\langle u_1,\phi\rangle}
\end{align*}
Using $(4.2)$, we have
\begin{small}
\begin{align*}
-\tilde{\lambda}\langle u_1,\phi\rangle = \langle A_{12}(A_{22}-\lambda)^{-1}A_{21}u_1, \phi\rangle +\tilde{\lambda}\langle A_{12}(A_{22}-\lambda)^{-2}A_{21}u_1, \phi\rangle+\tilde{\lambda}^2\langle A_{12}(A_{22}-\lambda)^{-3}A_{21}u_1, \phi\rangle+\cdots
\end{align*}
\end{small} 
namely 
\begin{align*}
-\tilde{\lambda}\langle u_1,\phi\rangle = \langle A_{12}(A_{22}-\lambda)^{-1}A_{21}u_1, \phi\rangle +\sum_{n=1}^\infty\tilde{\lambda}^n\langle A_{12}(A_{22}-\lambda)^{-n-1}A_{21}u_1, \phi\rangle
\end{align*}

Then by self-adjointness we have
\begin{align*}
-\tilde{\lambda}\langle u_1,\phi\rangle = \langle u_1, A_{12}(A_{22}-\lambda)^{-1}A_{21}\phi\rangle+\sum_{n=1}^\infty\tilde{\lambda}^n\langle u_1, A_{12}(A_{22}-\lambda)^{-n-1}A_{21}\phi\rangle
\end{align*}
Thus 
\begin{align*}
-\tilde{\lambda}=\frac{\langle u_1, A_{12}(A_{22}-\lambda)^{-1}A_{21}\phi\rangle}{\langle u_1,\phi\rangle}+\sum_{n=1}^\infty\tilde{\lambda}^n\frac{\langle u_1, A_{12}(A_{22}-\lambda)^{-n-1}A_{21}\phi\rangle}{\langle u_1,\phi\rangle}
\end{align*}
\end{proof}

For later discussions, let 
$ a_0=-\frac{\langle u_1, A_{12}(A_{22}-\lambda)^{-1}A_{21}\phi\rangle}{\langle u_1,\phi\rangle}, a_n=-\frac{\langle u_1, A_{12}(A_{22}-\lambda)^{-n-1}A_{21}\phi\rangle}{\langle u_1,\phi\rangle}$ and $g(x)=a_0+\sum_{n=1}^\infty a_nx^n$. Then we saw in Lemma 4.2 that the equation for $\tilde{\lambda}$ is just
\begin{align*}
\tilde{\lambda}=g(\tilde{\lambda})
\end{align*}
Clearly $\tilde{\lambda}$ is a fixed point of the map $g(x)$. And in the next section we are going to show $g(x)$ is a contraction map and as a corallary we will have an iterative scheme for solving $\tilde{\lambda}.$

\subsection{Iterative Scheme for Solving $\tilde{\lambda}$}
In this section we will show $g(x)=a_0+\sum_{n=1}^\infty a_nx^n$ 
 is a contraction map. And we will also give the iterative scheme for solving $\tilde{\lambda}.$

\begin{lemma}
$||u_1-\phi||_{\mathbf{L}^2}=O(\epsilon^{3\alpha_1})$
\end{lemma}
\begin{proof}
See \cite{LS} page 5.
\end{proof}

\begin{corollary} As $\epsilon\rightarrow 0$, we have $|\langle u_1,\phi\rangle|\geq \frac{1}{2}||u_1||\cdot||\phi||.$
\end{corollary}
\begin{proof}
Directly follows from Lemma 4.3.
\end{proof}

Before we do the main estimates about the coefficients $a_n$ we want to understand the operator $A_{21}$ in more details.

\begin{lemma}
Let $f(x,y)=\chi(x)\sqrt{\frac{2}{\epsilon h(x)}}\sin(\frac{\pi y}{\epsilon h(x)})\in \mathcal{L}_\epsilon,$  then for some pure constant $C$ and $D$ we have 

\begin{align*}
||A_{21}f(x,y)||^2\leq C\int_{\mathbf{I}}\chi'^2dx+D\int_{\mathbf{I}}\chi^2dx
\end{align*}
\end{lemma}
\begin{proof}
By direct computation and see Appendix.
\end{proof}

Now we prove a key Lemma which played an essential role in estimating the coefficients $a_n$.

\begin{lemma}

$\frac{||A_{21}u_1||\cdot||A_{21}\phi||}{||\langle u_1,\phi\rangle||}=O(\frac{1}{\epsilon^{2\alpha_1}})$.
\end{lemma}
\begin{proof}

Notice $\phi\in \mathcal{L}_\epsilon$, we can let $\phi=\chi(x)\sqrt{\frac{2}{\epsilon h(x)}}\sin(\frac{\pi y}{\epsilon h(x)})$ for some $\chi(x)$. Then $A_{11}\phi=\lambda \phi$ is equivalent to  $$ [-\frac{d^2}{dx^2}+\frac{\pi^2}{\epsilon^2h^2}+(\frac{\pi^2}{3}+\frac{1}{4})\frac{h'^2}{h^2}]\chi=\lambda\chi(x)$$
namely $-\chi''=[\lambda-\frac{\pi^2}{\epsilon^2h^2}-(\frac{\pi^2}{3}+\frac{1}{4})\frac{h'^2}{h^2}]\chi$.
In particular, it implies that 
\begin{align*}
\int \chi'^2 dx=\int [\lambda-\frac{\pi^2}{\epsilon^2h^2}-(\frac{\pi^2}{3}+\frac{1}{4})\frac{h'^2}{h^2}]\chi^2 dx= O(\frac{\mu}{\epsilon^{2\alpha_1}})\int \chi^2 dx
\end{align*}
since we have shown $\lambda-\frac{\pi^2}{\epsilon^2M^2}\sim \frac{\mu}{\epsilon^{2\alpha_1}}$ in Theorem 3.19 (One can also refer to the two term asymptotics in  \cite{LS}) where $\mu$ are eigenvalues of the operator on $L^2(\R)$ given by $-\frac{d^2}{dx^2}+q(x)$ where $q(x)=2\pi^2M^{-3}cx^m.$
So from Lemma 4.5, we have 
\begin{align*}
\frac{||A_{21}\phi||}{||\phi||}=O(\frac{1}{\epsilon^{\alpha_1}})
\end{align*}


Indeed we just showed that all the eigenfunctions of $A_{11}$ would have similar estimate. In fact, let's assume $\{\xi_j\}_{j=0}^\infty$ be all the normalized eigenfunctions of $A_{11}$ with corresponding eigenvalues $\{\lambda_j\}.$ Then 
\begin{align}
\frac{||A_{21}\xi_j||}{||\xi_j||}=O(\frac{\sqrt{\mu_j}}{\epsilon^{\alpha_1}})=O(\frac{\sqrt{\lambda_j}}{\epsilon^{\alpha_1}})
\end{align}

But we know $\{\xi_j\}$ form a complete basis for $\mathcal{L}_\epsilon.$ In particular, this allows us to show $\frac{||A_{21}u_1||}{||u_1||}=O(\frac{1}{\epsilon^{\alpha_1}})$. Indeed

 Let $x_j=\langle u_1, \xi_j\rangle$. Then $u_1=\sum_{j=0}^\infty x_j\xi_j$, $||u_1||^2=\sum_j x_j^2<\infty,$ and we also have $ ||A_{11}^2u_1||^2=\sum_{j=0}^\infty (\lambda_j^2 x_j)^2=\sum_{j=0}^\infty \lambda_j^4x_j^2<\infty$.  Moreover,
\begin{align*}
||A_{21}u_1||^2&=||\sum_{j}x_j A_{21}\xi_j||^2=\langle \sum_{j}x_j A_{21}\xi_j, \sum_{k}x_j A_{21}\xi_k\rangle\\
&=\sum_j x_j^2||A_{21}\xi_j||^2+\sum_{j\neq k}x_jx_k\langle A_{21}\xi_j, A_{21}\xi_k\rangle\\
&\leq D\frac{1}{\epsilon^{2\alpha_1}}\left[ \sum_j x_j^2\mu_j+\sum_{j\neq k}|\sqrt{\mu_j} x_j|\cdot|\sqrt{\mu_k} x_k|\right]\\
&\leq D\frac{1}{\epsilon^{2\alpha_1}}\left(\sum_{j}|\sqrt{\mu_j} x_j|\right)^2\\
&\leq D\frac{1}{\epsilon^{2\alpha_1}}\left(\sum_{j}|\mu_j^2 x_j\cdot\frac{1}{\mu_j\sqrt{\mu_j}}|\right)^2\\
&\leq D\frac{1}{\epsilon^{2\alpha_1}}\left(\sum_{j}[\frac{|\mu_j^2 x_j|^p}{p}+\frac{1}{q}\frac{1}{\mu_j^{\frac{3q}{2}}}]\right)^2= D\frac{1}{\epsilon^{2\alpha_1}}\left(\sum_{j}\frac{|\mu_j^2 x_j|^p}{p}+\sum_{j}\frac{1}{q}\frac{1}{\mu_j^{\frac{3q}{2}}}\right)^2
\end{align*}
where $p, q$ are any positive conjugates, namely $\frac{1}{p}+\frac{1}{q}=1.$ And in the last inequality we are using classical inequality:
$uv\leq \frac{u^p}{p}+\frac{v^q}{v}$ for $u,v>0.$
For our argument, let's fix $p=4, q=\frac{4}{3}.$

So we just showed
\begin{align*}
||A_{21}u_1||&\leq D\frac{1}{\epsilon^{\alpha_1}}\left(\sum_{j}\frac{|\mu_j^2 x_j|^4}{4}+\sum_{j}\frac{3}{4}\frac{1}{\mu_j^2}\right)=D\frac{1}{\epsilon^{\alpha_1}}\left(\frac{1}{4}\sum_{j}|\mu_j^2 x_j|^4+\frac{3}{4}\sum_{j}\frac{1}{\mu_j^2}\right)
\end{align*}

But notice $\mu_j\leq \epsilon^{2\alpha_1}\lambda_j$ and $\sum_j|\lambda_j^2x_j|^4<\infty$ as $\sum_{j=0}^\infty (\lambda_j^2 x_j)^2=||A_{11}^2u_1||^2<\infty$, 
So 
\begin{align*}
\sum_{j}|\mu_j^2 x_j|^4\leq O(\epsilon^{16\alpha_1})\sum_j|\lambda_j^2x_j|^4<\infty
\end{align*}

Also notice  $\sum_j \frac{1}{\lambda_j^2}<\infty$ since $\frac{1}{\lambda_j}\sim \frac{1}{\mu_j}\sim \frac{1}{j}$. 


In conclusion, for $\forall\gamma>0, \exists N$, such that $\frac{1}{4}\sum_{j=N}^\infty|\mu_j^2x_j|^4<\gamma.$ In particular, this implies $\frac{1}{4}\sum_{j=N}^\infty|\mu_j^2x_j|^4+\frac{3}{4}\sum_{j=0}^\infty \frac{1}{\mu_j^2}\leq \Gamma$ for some constant $\Gamma$ which does not depend on $\epsilon.$
Thus $$||A_{21}u_1||\leq D\frac{1}{\epsilon^{\alpha_1}}\left(\Gamma+\sum_{j=0}^{N-1}\frac{1}{4}|\mu_j^2 x_j|^4 \right)$$
Now let $\mu_*=\max\{\mu_0, \mu_1,\cdots, \mu_{N-1}\}$. Clearly $\mu_*$ does not depend on $\epsilon.$
 Then 
$$||A_{21}u_1||\leq D\frac{1}{\epsilon^{\alpha_1}}\Gamma+\frac{D}{4}\frac{1}{\epsilon^{\alpha_1}}\mu_*^8\sum_{j=0}^{N-1}|x_j|^4\leq O(\frac{1}{\epsilon^{\alpha_1}})||u_1||$$

In particualr, we have $\frac{||A_{21}u_1||}{||u_1||}=O(\frac{1}{\epsilon^{\alpha_1}})$. 

To finish the proof, one simply need to notice in Corollary 4.4 we have  $$|\langle u_1,\phi\rangle|\geq \frac{1}{2}||u_1||\cdot||\phi||$$

And in summary, we just showed that 
$\frac{||A_{21}u_1||\cdot||A_{21}\phi||}{||\langle u_1,\phi\rangle||}=O(\frac{1}{\epsilon^{2\alpha_1}})$.

\end{proof}

Now we are ready to show that $g(x)=a_0+\sum_{n=1}^\infty a_nx^n$ is a contraction.

\begin{theorem}
$g(x)=a_0+\sum_{n=1}^\infty a_nx^n$ where  $$ a_0=-\frac{\langle u_1, A_{12}(A_{22}-\lambda)^{-1}A_{21}\phi\rangle}{\langle u_1,\phi\rangle}, a_n=-\frac{\langle u_1, A_{12}(A_{22}-\lambda)^{-n-1}A_{21}\phi\rangle}{\langle u_1,\phi\rangle}$$ is a contraction.

Notice
\begin{align*}
|a_0| &=\frac{|\langle u_1, A_{12}(A_{22}-\lambda)^{-1}A_{21}\phi\rangle|}{|\langle u_1,\phi\rangle|} =\frac{|\langle A_{21}u_1, (A_{22}-\lambda)^{-1}A_{21}\phi\rangle|}{|\langle u_1,\phi\rangle|}\\
&\leq \frac{||A_{21}u_1||\cdot||A_{21}\phi||}{||\langle u_1,\phi\rangle||}||(A_{22}-\lambda)^{-1}||\\
|a_n| &=\frac{|\langle u_1, A_{12}(A_{22}-\lambda)^{-n-1}A_{21}\phi\rangle|}{|\langle u_1,\phi\rangle|} =\frac{|\langle A_{21}u_1, (A_{22}-\lambda)^{-n-1}A_{21}\phi\rangle|}{|\langle u_1,\phi\rangle|}\\
&\leq \frac{||A_{21}u_1||\cdot||A_{21}\phi||}{||\langle u_1,\phi\rangle||}||(A_{22}-\lambda)^{-n-1}||
\end{align*}

Easy to see that $\frac{\langle A_{22}v, v\rangle	}{\langle v, v\rangle }\geq \frac{4\pi^2}{\epsilon^2M^2}, \forall v\in \mathrm{L}^2(\Omega_\epsilon).$

 This implies $||(A_{22}-\lambda)^{-1}||\leq C\epsilon^2, ||(A_{22}-\lambda)^{-n-1}||\leq (C\epsilon^2)^{(n+1)}$ for some constant $C$, which does not depend on $\epsilon$, in particlular we have  
 
 \begin{align*}
 |a_0|&\leq C\frac{||A_{21}u_1||\cdot||A_{21}\phi||}{||\langle u_1,\phi\rangle||}\epsilon^2\\
 |a_n|&\leq \frac{||A_{21}u_1||\cdot||A_{21}\phi||}{||\langle u_1,\phi\rangle||}(C\epsilon^2)^{(n+1)}
 \end{align*}
 
 \textbf{Claim :}
\begin{enumerate}
\item $g(x)$ has convergence radius less than $C\epsilon^2.$
\item $g'(x)=\sum_{n=1}^\infty na_nx^{n-1}$ satisfy $|g'(x)|<\frac{1}{2}$ for any $x$ inside the radius of convergence when $\epsilon$ is small enough.
\end{enumerate}

\textbf{Proof of the Claim:}
\begin{enumerate}
\item $\lim_{n\rightarrow \infty}\sqrt[n]{|\frac{a_n}{a}|}\leq C\epsilon^2.$
\item 
\begin{align*}
|g'(x)|&\leq \frac{||A_{21}u_1||\cdot||A_{21}\phi||}{||\langle u_1,\phi\rangle||} \sum_{n=1}^\infty n(C\epsilon^2)^{n+1}|x|^{n-1}\leq \frac{||A_{21}u_1||\cdot||A_{21}\phi||}{||\langle u_1,\phi\rangle||} \sum_{n=1}^\infty n(C\epsilon^2)^{2n}\\
&\leq \frac{||A_{21}u_1||\cdot||A_{21}\phi||}{||\langle u_1,\phi\rangle||} \frac{C^2\epsilon^4}{(1-c^2\epsilon^4)^2}
\end{align*}
Now using Lemma 4.6 we have 
\begin{align*}
|g'(x)|\leq O(\frac{1}{\epsilon^{2\alpha_1}})\frac{C^2\epsilon^4}{(1-c^2\epsilon^4)^2}
\end{align*}
So $|g'(x)|\rightarrow 0$ as $\epsilon\rightarrow 0.$
\end{enumerate}

With the claim we see that $g(x)$ is indeed a contraction.
\end{theorem}

\begin{corollary}
Let $\lambda$ be eigenvalues of $A_{11}$ with normalized eigenfunction $\phi$. We also let  $\tilde{\lambda}=\Lambda_\epsilon-\lambda$ . Then $\tilde{\lambda}_n\rightarrow \tilde{\lambda}$ as $n\rightarrow \infty,$ where
\begin{align*}
\tilde{\lambda}_0&=a_0\\
\tilde{\lambda}_{n+1}&=g(\tilde{\lambda}_n)\\
\end{align*}
 $ a_0=-\frac{\langle u_1, A_{12}(A_{22}-\lambda)^{-1}A_{21}\phi\rangle}{\langle u_1,\phi\rangle}, a_n=-\frac{\langle u_1, A_{12}(A_{22}-\lambda)^{-n-1}A_{21}\phi\rangle}{\langle u_1,\phi\rangle}$ and $g(x)=a_0+\sum_{n=1}^\infty a_nx^n$.
\end{corollary}
\begin{proof}
Recall Lemma 4.2 where we show that $\tilde{\lambda}=g(\tilde{\lambda}).$ So the corollary follows directly from Theorem 4.7 where we showed that $g(x)$ is a contraction.
\end{proof}


\bibliographystyle{amsplain}
\section*{Acknowledgement}
The author would like to thank Prof Leonid Friedlander for a lot of  helpful discussions  while preparing this paper as part of the dissertation work.

\section*{Appendix of Proof of Lemmas}
\subsection{Proof of Lemma 3.6}
\begin{proof}
\begin{tiny}
\begin{align*}
&\lambda(H_{\epsilon,K}) 
\\
&= \frac{1}{2\pi\mathfrak{i}}\mathbf{Tr}\int_{\Gamma}\lambda (H_{\epsilon,K}-\lambda)^{-1}d\lambda\\
 &= \frac{1}{2\pi\mathfrak{i}}\mathbf{Tr}\int_{\Gamma}\lambda (H_{\epsilon,K}-H_0+H_0-\lambda)^{-1}d\lambda\\
& = \frac{1}{2\pi\mathfrak{i}}\mathbf{Tr}\int_{\Gamma} \lambda (H_{0}-\lambda)^{-1}\left[ \sum_{k=0}^N (-1)^k[\left(  H_{\epsilon,K}-H_0 \right)(H_0-\lambda)^{-1}]^k+(-1)^{(N+1)}\left(I+[(H_{\epsilon,K}-H_0)(H_0-\lambda)^{-1}]\right)^{-1}[(H_{\epsilon,K}-H_0)(H_0-\lambda)^{-1}]^{N+1}\right]
d\lambda \\
& =\sum_{k=0}^N \frac{1}{2\pi\mathfrak{i}} \mathbf{Tr}\int_{\Gamma}(-1)^k\lambda (H_{0}-\lambda)^{-1}
\left(\sum_{n=k}^\infty \epsilon^{n\alpha} \sum_{j_1+j_2+\cdots+j_k=n, j_i\leq K, j_i\in \Z^{+}}H_{j_1}(H_0-\lambda)^{-1}H_{j_2}(H_0-\lambda)^{-1}\cdots H_{j_k}(H_0-\lambda)^{-1}\right)d\lambda        \\  
 &+\sum_{n=N+1}^\infty \epsilon^{n\alpha}\frac{1}{2\pi\mathfrak{i} }\mathbf{Tr}\int_{\Gamma}\sum_{j_1+j_2+\cdots+j_{N+1}=n, j_i\leq K, j_i\in \Z^{+}}(-1)^{N+1}\lambda (H_{\epsilon,K}-\lambda)^{-1}H_{j_1}(H_0-\lambda)^{-1}H_{j_2}(H_0-\lambda)^{-1}\cdots H_{j_{N+1}}(H_0-\lambda)^{-1} d\lambda\\
& =\mu_0+\sum_{n=1}^N\tilde{q}_n\epsilon^{n\alpha}\\
& +\sum_{n=N+1}^\infty \epsilon^{n\alpha}\sum_{k=1}^N\frac{1}{2\pi\mathfrak{i} }\mathbf{Tr}\int_{\Gamma}\left(\sum_{j_1+j_2+\cdots+j_k=n, j_i\leq K ,j_i\in \Z^{+}}(-1)^k\lambda (H_{0}-\lambda)^{-1}
H_{j_1}(H_0-\lambda)^{-1}H_{j_2}(H_0-\lambda)^{-1}\cdots H_{j_k}(H_0-\lambda)^{-1} \right)d\lambda
\\
&+\sum_{n=N+1}^\infty \epsilon^{n\alpha}\frac{1}{2\pi\mathfrak{i} }\mathbf{Tr}\int_{\Gamma}\sum_{j_1+j_2+\cdots+j_{N+1}=n, j_i\leq K ,j_i\in \Z^{+}}(-1)^{N+1}\lambda (H_{\epsilon,K}-\lambda)^{-1}
H_{j_1}(H_0-\lambda)^{-1}H_{j_2}(H_0-\lambda)^{-1}\cdots H_{j_{N+1}}(H_0-\lambda)^{-1} d\lambda\\
\end{align*}
\end{tiny}

Hence we have 
\begin{tiny}
\begin{align*}
 &\lambda(H_{\epsilon,K}) -\left(\mu_0+\sum_{n=1}^N\tilde{q}_n\epsilon^{n\alpha}\right) \\
&=\sum_{n=N+1}^\infty \epsilon^{n\alpha}\sum_{k=1}^N\frac{1}{2\pi\mathfrak{i} }\mathbf{Tr}\int_{\Gamma}\left(\sum_{j_1+j_2+\cdots+j_k=n, j_i\leq K ,j_i\in \Z^{+}}(-1)^k\lambda (H_{0}-\lambda)^{-1}
H_{j_1}(H_0-\lambda)^{-1}H_{j_2}(H_0-\lambda)^{-1}\cdots H_{j_k}(H_0-\lambda)^{-1} \right)d\lambda
\\
&+\sum_{n=N+1}^\infty \epsilon^{n\alpha}\frac{1}{2\pi\mathfrak{i} }\mathbf{Tr}\int_{\Gamma}\sum_{j_1+j_2+\cdots+j_{N+1}=n, j_i\leq K ,j_i\in \Z^{+}}(-1)^{N+1}\lambda (H_{\epsilon,K}-\lambda)^{-1}
H_{j_1}(H_0-\lambda)^{-1}H_{j_2}(H_0-\lambda)^{-1}\cdots H_{j_{N+1}}(H_0-\lambda)^{-1} d\lambda\\
&=\sum_{n=N+1}^{KN} \epsilon^{n\alpha}\sum_{k=1}^N\frac{1}{2\pi\mathfrak{i} }\mathbf{Tr}\int_{\Gamma=\{|\lambda-\lambda_0|\leq \epsilon\}}\left(\sum_{j_1+j_2+\cdots+j_k=n, j_i\leq K ,j_i\in \Z^{+}}(-1)^k\lambda (H_{0}-\lambda)^{-1}
H_{j_1}(H_0-\lambda)^{-1} H_{j_2}(H_0-\lambda)^{-1} \cdots H_{j_k}(H_0-\lambda)^{-1} \right)d\lambda
\\
& +\sum_{n=N+1}^{K(N+1)} \epsilon^{n\alpha}\frac{1}{2\pi\mathfrak{i} }\mathbf{Tr}\int_{\Gamma=\{|\lambda-\lambda_0|\leq \epsilon\}}\sum_{j_1+j_2+\cdots+j_{N+1}=n, j_i\leq K ,j_i\in \Z^{+}}(-1)^{N+1}\lambda (H_{\epsilon,K}-\lambda)^{-1}
H_{j_1}(H_0-\lambda)^{-1} H_{j_2}(H_0-\lambda)^{-1} \cdots H_{j_{N+1}}(H_0-\lambda)^{-1} d\lambda\\
&= O(\epsilon^{(N+1)\alpha})
\end{align*}
\end{tiny}

The main reason for the last equality is that $H_{\epsilon, K}$ is analytic family of type B perturbation of $H_0$.
\end{proof}

\subsection{Proof of Lemma 3.13}

\begin{proof}
Taylor Expansion. More precisely, notice  that $h(x)=M-cx^m$ and $\frac{\pi^2}{\epsilon^2 h^2}=\frac{\pi^2}{M^2\epsilon^2}\left[\sum_{n=1}^\infty n \left(\frac{cx^m}{M}\right)^{n-1}\right]$, so we have 

\begingroup\makeatletter\def\f@size{9}\check@mathfonts
\def\maketag@@@#1{\hbox{\m@th\large\normalfont#1}}
\begin{align*}
A_\epsilon &=-\frac{d^2}{dx^2}+\frac{\pi^2}{\epsilon^2h^2}-\frac{\pi^2}{\epsilon^2M^2}+(\frac{\pi^2}{3}+\frac{1}{4})\frac{h'^2}{h^2}\\
&=-\frac{d^2}{dx^2}+\frac{\pi^2}{M^2\epsilon^2}\left[\sum_{n=2}^\infty n \left(\frac{cx^m}{M}\right)^{n-1}\right]+(\frac{\pi^2}{3}+\frac{1}{4})\frac{m^2c^2x^{2(m-1)}+c'^2x^{2m}+2mcc'x^{2m-1}}{M^2}\left[\sum_{n=1}^\infty n \left(\frac{cx^m}{M}\right)^{n-1}\right]
\end{align*}\endgroup

Buy introducing $x=\epsilon^{\alpha_1}y$ where $\alpha_1=\frac{2}{m+2}, y\in \mathbf{I}_\epsilon=[-\frac{l_1}{\epsilon^{\alpha_1}},\frac{l_2}{\epsilon^{\alpha_1}}]$, we see that 
\begin{tiny}
\begin{align*}
\epsilon^{2\alpha_1} A_\epsilon 
& =-\frac{d^2}{dy^2}+2a_0a_1y^m+\sum_{n=1}^\infty\left[(n+2)a_0a_1^{n+1}y^{nm+m}+(n-1)aa_0a_1^{n-2}y^{nm-2}\right]\epsilon^{n\alpha}\\
 &\quad+\sum_{n=1}^\infty (n-1)b_1a_1^{n-2}y^{nm-1}\epsilon^{n\alpha+\alpha_1}+\sum_{n=1}^\infty (n-1)b_2a_1^{n-2}y^{nm}\epsilon^{n\alpha+2\alpha_1}\\
 &=-\frac{d^2}{dy^2}+2a_0\left(\sum_{n=0}^\infty \frac{c_n}{M}y^n\epsilon^{n\alpha_1}\right)y^m+\sum_{n=1}^\infty \epsilon^{nm\alpha_1}\left[(n+2)a_0\left(\sum_{k=0}^\infty \alpha_{k,n}y^k\epsilon^{k\alpha_1}\right)y^{nm+m} \right]\\
 &\quad  +\sum_{n=1}^\infty \epsilon^{nm\alpha_1}\left[(\frac{\pi^2}{3}+\frac{1}{4})\frac{m^2}{\pi^2}(n-1)a_0\left(\sum_{n=0}^\infty d_ny^n\epsilon^{n\alpha_1}\right)\left(\sum_{k=0}^\infty \beta_{k,n}y^k\epsilon^{k\alpha_1}\right) \right]\\
  &\quad  +\sum_{n=1}^\infty \epsilon^{nm\alpha_1+\alpha_1}\left[\frac{2m}{M^2}(\frac{\pi^2}{3}+\frac{1}{4})(n-1)\left(\sum_{n=0}^\infty f_ny^n\epsilon^{n\alpha_1}\right)\left(\sum_{k=0}^\infty \beta_{k,n}y^k\epsilon^{k\alpha_1}\right) \right]\\
   &\quad  +\sum_{n=1}^\infty \epsilon^{nm\alpha_1+2\alpha_1}\left[(\frac{\pi^2}{3}+\frac{1}{4})(n-1)\frac{1}{M^2}\left(\sum_{n=0}^\infty g_ny^n\epsilon^{n\alpha_1}\right)\left(\sum_{k=0}^\infty \beta_{k,n}y^k\epsilon^{k\alpha_1}\right) \right]\\
  &=-\frac{d^2}{dy^2}+2a_0\left(\sum_{n=0}^\infty \frac{c_n}{M}y^n\epsilon^{n\alpha_1}\right)y^m+\sum_{n=1}^\infty \epsilon^{nm\alpha_1}\left[(n+2)a_0\left(\sum_{k=0}^\infty \alpha_{k,n}y^k\epsilon^{k\alpha_1}\right)y^{nm+m} \right]\\
   &\quad  +\sum_{n=1}^\infty \epsilon^{nm\alpha_1}\left(\sum_{k=0}^\infty \beta_{k,n}y^k\epsilon^{k\alpha_1}\right) \left[A_1\left(\sum_{n=0}^\infty d_ny^n\epsilon^{n\alpha_1}\right)+A_2\epsilon^{\alpha_1}\left(\sum_{n=0}^\infty f_ny^n\epsilon^{n\alpha_1}\right)+A_3\epsilon^{2\alpha_1}\left(\sum_{n=0}^\infty g_ny^n\epsilon^{n\alpha_1}\right)\right]\\
   &=-\frac{d^2}{dy^2}+2a_0\left(\sum_{n=0}^\infty \frac{c_n}{M}y^n\epsilon^{n\alpha_1}\right)y^m+\sum_{n=1}^\infty \epsilon^{nm\alpha_1}\left[(n+2)a_0\left(\sum_{k=0}^\infty \alpha_{k,n}y^k\epsilon^{k\alpha_1}\right)y^{nm+m} \right]\\
   &\quad  +\sum_{n=1}^\infty \epsilon^{nm\alpha_1}\left(\sum_{k=0}^\infty \beta_{k,n}y^k\epsilon^{k\alpha_1}\right) \left[\sum_{k=0}^\infty \left(A_1d_ky^k+A_2f_{k-1}y^{k-1}+A_3g_{k-2}y^{k-2}\right)\epsilon^{k\alpha_1}\right]\\
     &=-\frac{d^2}{dy^2}+2a_0\left(\sum_{n=0}^\infty \frac{c_n}{M}y^n\epsilon^{n\alpha_1}\right)y^m+\sum_{n=1}^\infty \epsilon^{nm\alpha_1}\left[(n+2)a_0\left(\sum_{k=0}^\infty \alpha_{k,n}y^k\epsilon^{k\alpha_1}\right)y^{nm+m} \right]\\
   &\quad  +\sum_{n=1}^\infty \epsilon^{nm\alpha_1}\left(\sum_{k=0}^\infty \beta_{k,n}y^k\epsilon^{k\alpha_1}\right) \left[\sum_{k=0}^\infty \gamma_k\epsilon^{k\alpha_1}\right]\\
    &=-\frac{d^2}{dy^2}+2a_0\left(\sum_{n=0}^\infty \frac{c_n}{M}y^n\epsilon^{n\alpha_1}\right)y^m+\sum_{n=1}^\infty \epsilon^{nm\alpha_1}\left[(n+2)a_0\left(\sum_{k=0}^\infty \alpha_{k,n}y^k\epsilon^{k\alpha_1}\right)y^{nm+m} \right]\\
   &\quad  +\sum_{n=1}^\infty \epsilon^{nm\alpha_1}\sum_{k=0}^\infty \left(\sum_{i+j=k,i,j\in \mathbf{N}}\beta_{i,n}y^i\gamma_j\right)\epsilon^{k\alpha_1}\\
    &=-\frac{d^2}{dy^2}+2a_0\left(\sum_{n=0}^\infty \frac{c_n}{M}y^n\epsilon^{n\alpha_1}\right)y^m\\
   &\quad  +\sum_{n=1}^\infty \epsilon^{nm\alpha_1}\sum_{k=0}^\infty \left(\left[(n+2)a_0\alpha_{k,n}y^{k+nm+m}+\sum_{i+j=k,i,j\in \mathbf{N}}\beta_{i,n}y^i\gamma_j\right)\right]\epsilon^{k\alpha_1}\\
     &=-\frac{d^2}{dy^2}+\left(\sum_{n=0}^\infty \frac{2a_0c_n}{M}y^{n+m}\epsilon^{n\alpha_1}\right)\\
   &\quad  +\sum_{n=m}^\infty \sum_{k+sm=n,k\in \mathbf{N},s\geq 1}\left[(s+2)a_0\alpha_{k,s}y^{n+m}+\sum_{i+j=k,i,j\in \mathbf{N}}\beta_{i,s}y^i\gamma_j\right]\epsilon^{n\alpha_1}\\
  &= -\frac{d^2}{dy^2}+\frac{2a_0c_0}{M}y^m+\sum_{n=1}^\infty H_n\epsilon^{n\alpha_1}
\end{align*}
\end{tiny}
with $ b_1=(\frac{\pi^2}{3}+\frac{1}{4})\frac{2mc(\epsilon^{\alpha_1}y)c'(\epsilon^{\alpha_1}y)}{M^2},  
 b_2=(\frac{\pi^2}{3}+\frac{1}{4})\frac{c'^2(\epsilon^{\alpha_1}y)}{M^2}$,  $\alpha=m\alpha_1=\frac{2m}{m+2}, a_0=\frac{\pi^2}{M^2}, a_1=\frac{c(\epsilon^{\alpha_1}y)}{M} a=(\frac{\pi^2}{3}+\frac{1}{4})\frac{m^2c^2(\epsilon^{\alpha_1}y)}{\pi^2}$   and 
 \begin{align*}
H_n=\frac{2a_0c_n}{M}y^{n+m}+\sum_{k+sm=n,k\in \mathbf{N},s\geq 1} \left[(s+2)a_0\alpha_{k,s}y^{n+m}+\sum_{i+j=k,i,j\in \mathbf{N}}\beta_{i,s}y^i\gamma_j\right]
\end{align*}
where $\gamma_j=A_1d_jy^j+A_2f_{j-1}y^{j-1}+A_3g_{j-2}y^{j-2}$.
\end{proof}

\subsection{Proof of Lemma 3.14}
\begin{proof}
\begin{tiny}
\begin{align*}
&\lambda(H_{\epsilon,K}) 
\\
&= \frac{1}{2\pi\mathfrak{i}}\mathbf{Tr}\int_{\Gamma}\lambda (H_{\epsilon,K}-\lambda)^{-1} d\lambda \\
 &= \frac{1}{2\pi\mathfrak{i}}\mathbf{Tr}\int_{\Gamma}\lambda \left(H_{\epsilon,K}-H_0+H_0-\lambda\right)^{-1} d\lambda \\
& = \frac{1}{2\pi\mathfrak{i}}\mathbf{Tr}\int_{\Gamma} \lambda (H_0-\lambda)^{-1}\left[ \sum_{k=0}^N (-1)^k[\left(  H_{\epsilon,K}-H_0 \right)(H_0-\lambda)^{-1}]^k+(-1)^{(N+1)}\left( I+[(H_{\epsilon,K}-H_0)(H_0-\lambda)^{-1}]\right)^{-1}[(H_{\epsilon,K}-H_0)(H_0-\lambda)^{-1}]^{N+1}
\right] d\lambda \\
& =\sum_{k=0}^N \frac{1}{2\pi\mathfrak{i}} \mathbf{Tr}\int_{\Gamma} (-1)^k\lambda(H_0-\lambda)^{-1}\left(\sum_{n=k}^\infty \epsilon^{n\alpha_1} \sum_{j_1+j_2+\cdots+j_k=n, j_i\leq K, j_i\in \Z^{+}}H_{j_1}(H_0-\lambda)^{-1}H_{j_2}(H_0-\lambda)^{-1}\cdots H_{j_k}(H_0-\lambda)^{-1}\right)d\lambda        \\  
 &+\sum_{n=N+1}^\infty \epsilon^{n\alpha_1}\frac{1}{2\pi\mathfrak{i} }\mathbf{Tr}\int_{\Gamma}\sum_{j_1+j_2+\cdots+j_{N+1}=n, j_i\leq K, j_i\in \Z^{+}}(-1)^{N+1}\lambda(H_{\epsilon,K}-\lambda)^{-1}H_{j_1}(H_0-\lambda)^{-1}H_{j_2}(H_0-\lambda)^{-1}\cdots H_{j_{N+1}}(H_0-\lambda)^{-1} d\lambda\\
& =\mu_0+\sum_{n=1}^N\tilde{q}_n\epsilon^{n\alpha_1}\\
& +\sum_{n=N+1}^\infty \epsilon^{n\alpha_1}\sum_{k=1}^N\frac{1}{2\pi\mathfrak{i} }\mathbf{Tr}\int_{\Gamma}\left(\sum_{j_1+j_2+\cdots+j_k=n, j_i\leq K ,j_i\in \Z^{+}}
(-1)^k\lambda(H_0-\lambda)^{-1}
H_{j_1}(H_0-\lambda)^{-1}H_{j_2}(H_0-\lambda)^{-1}\cdots H_{j_k}(H_0-\lambda)^{-1} \right)d\lambda
\\
&+\sum_{n=N+1}^\infty \epsilon^{n\alpha_1}\frac{1}{2\pi\mathfrak{i} }\mathbf{Tr}\int_{\Gamma}\sum_{j_1+j_2+\cdots+j_{N+1}=n, j_i\leq K ,j_i\in \Z^{+}}
(-1)^{N+1}\lambda(H_{\epsilon,K}-\lambda)^{-1}
H_{j_1}(H_0-\lambda)^{-1}H_{j_2}(H_0-\lambda)^{-1}\cdots H_{j_{N+1}}(H_0-\lambda)^{-1} d\lambda\\
\end{align*}
\end{tiny}

Hence we have 
\begin{tiny}
\begin{align*}
 &\lambda(H_{\epsilon,K}) -\left(\mu_0+\sum_{n=1}^N\tilde{q}_n\epsilon^{n\alpha_1}\right) \\
&=\sum_{n=N+1}^\infty \epsilon^{n\alpha_1}\sum_{k=1}^N\frac{1}{2\pi\mathfrak{i} }\mathbf{Tr}\int_{\Gamma}\left(\sum_{j_1+j_2+\cdots+j_k=n, j_i\leq K ,j_i\in \Z^{+}}
(-1)^k\lambda(H_0-\lambda)^{-1}
H_{j_1}(H_0-\lambda)^{-1}H_{j_2}(H_0-\lambda)^{-1}\cdots H_{j_k}(H_0-\lambda)^{-1} \right)d\lambda
\\
&+\sum_{n=N+1}^\infty \epsilon^{n\alpha_1}\frac{1}{2\pi\mathfrak{i} }\mathbf{Tr}\int_{\Gamma}\sum_{j_1+j_2+\cdots+j_{N+1}=n, j_i\leq K ,j_i\in \Z^{+}}
(-1)^{N+1}\lambda(H_{\epsilon,K}-\lambda)^{-1}
H_{j_1}(H_0-\lambda)^{-1}H_{j_2}(H_0-\lambda)^{-1}\cdots H_{j_{N+1}}(H_0-\lambda)^{-1} d\lambda\\
&=\sum_{n=N+1}^{KN} \epsilon^{n\alpha_1}\sum_{k=1}^N\frac{1}{2\pi\mathfrak{i} }\mathbf{Tr}\int_{\Gamma=\{|\lambda-\lambda_0|\leq \epsilon\}}\left(\sum_{j_1+j_2+\cdots+j_k=n, j_i\leq K ,j_i\in \Z^{+}}
(-1)^k\lambda(H_0-\lambda)^{-1}
H_{j_1}(H_0-\lambda)^{-1} H_{j_2}(H_0-\lambda)^{-1} \cdots H_{j_k}(H_0-\lambda)^{-1} \right)d\lambda
\\
& +\sum_{n=N+1}^{K(N+1)} \epsilon^{n\alpha_1}\frac{1}{2\pi\mathfrak{i} }\mathbf{Tr}\int_{\Gamma=\{|\lambda-\lambda_0|\leq \epsilon\}}\sum_{j_1+j_2+\cdots+j_{N+1}=n, j_i\leq K ,j_i\in \Z^{+}}
(-1)^{N+1}\lambda(H_{\epsilon,K}-\lambda)^{-1}
H_{j_1}(H_0-\lambda)^{-1}  \cdots H_{j_{N+1}}(H_0-\lambda)^{-1} d\lambda\\
&= O(\epsilon^{(N+1)\alpha_1})
\end{align*}
\end{tiny}


The main reason for the last equality is that $H_{\epsilon, K}$ is analytic family of type B perturbation of $H_0$.
\end{proof}

\subsection{Proof of Lemma 4.5}
\begin{lemma}
\textbf{(Explicit Computation of $A_{21}$)} 

Let $f(x,y)=\chi(x)\sqrt{\frac{2}{\epsilon h(x)}}\sin(\frac{\pi y}{\epsilon h(x)})\in \mathcal{L}_\epsilon,$ we also let $g(x,y)=\sqrt{\frac{2}{\epsilon h(x)}}\sin(\frac{\pi y}{\epsilon h(x)})$, then 
\begingroup\makeatletter\def\f@size{9}\check@mathfonts
\def\maketag@@@#1{\hbox{\m@th\large\normalfont#1}}%
\begin{align*}
||A_{21}f(x,y)||^2 &=\int_\mathbf{I}4\left(\frac{\pi^2}{3}+\frac{1}{4}\right)\frac{h'^2}{h^2}\chi'^2dx +\int_\mathbf{I} 4\chi\chi'\left[\frac{h'}{h}\frac{h''}{h}\left(\frac{\pi^2}{3}+\frac{1}{4}\right)+\left(\frac{h'}{h}\right)^3\left(-\frac{1}{4}-\frac{4}{3}\pi^2\right)\right]dx\\
&\quad +\int_\mathbf{I} \chi^2\left[\left(\frac{1}{2}+\frac{53\pi^2}{12}+\frac{\pi^4}{5}\right)\left(\frac{h'}{h}\right)^4-\left(\frac{1}{2}+\frac{8\pi^2}{3}\right)\left(\frac{h'}{h}\right)^2\frac{h''}{h}
\right]dx\\
&\quad +\int_\mathbf{I}\chi^2\left[\left(\frac{\pi^2}{3}+\frac{1}{4}\right)\left(\frac{h''}{h}\right)^2-\left(\frac{\pi^2}{3}+\frac{1}{4}\right)^2\left(\frac{h'}{h}\right)^4+\left(\frac{1}{16}+\frac{1}{12}\pi^2\right)\left(\frac{h'}{h}\right)^3\right]dx
\end{align*}\endgroup
\end{lemma}
\begin{proof}
By computation we have 
\begin{enumerate}
\item $\Delta f=\chi''g+\chi g''+2\chi'g'+\chi \left(\frac{\pi }{\epsilon h}\right)^2g$
\item
\begin{align*}
P\Delta f&= \left[\int_0^{\epsilon h} \left(\chi''g+\chi g''+2\chi'g'+\chi \left(\frac{\pi }{\epsilon h}\right)^2g\right) gdy \right]g\\
&=\left[\chi''+\chi\left(\frac{\pi}{\epsilon h}\right)^2-\chi\left(\frac{\pi^2}{3}+\frac{1}{4}\right)\frac{h'^2}{h^2}\right]g
\end{align*}

\item 
\begin{align*}
||P\Delta f||^2&=\int_{\Omega_\epsilon} \left[\chi''+\chi\left(\frac{\pi}{\epsilon h}\right)^2-\chi\left(\frac{\pi^2}{3}+\frac{1}{4}\right)\frac{h'^2}{h^2}\right]^2g^2 dxdy\\
&=\int_{\mathbf{I}} \left[\chi''+\chi\left(\frac{\pi}{\epsilon h}\right)^2-\chi\left(\frac{\pi^2}{3}+\frac{1}{4}\right)\frac{h'^2}{h^2}\right]^2 dx
\end{align*}

\item we also have 
\begingroup\makeatletter\def\f@size{9}\check@mathfonts
\def\maketag@@@#1{\hbox{\m@th\large\normalfont#1}}%
\begin{align*}
||\Delta f||^2&=\langle \chi''g+\chi g''+2\chi'g'+\chi \left(\frac{\pi }{\epsilon h}\right)^2g, \chi''g+\chi g''+2\chi'g'+\chi \left(\frac{\pi }{\epsilon h}\right)^2g\rangle\\
&= \int_\mathbf{I} \left[\chi''+\chi\left(\frac{\pi}{\epsilon h}\right)^2\right]^2dx-2\int_\mathbf{I}\left(\frac{\pi^2}{3}+\frac{1}{4}\right)\frac{h'^2}{h^2}\chi\left[\chi''+\chi\left(\frac{\pi}{\epsilon h}\right)^2\right]dx+\int_\mathbf{I}4\chi'^2\left(\int_0^{\epsilon h}g'^2dy\right)dx\\
& \quad +\int_\mathbf{I}4\chi\chi'\left(\int_0^{\epsilon h}g'g''dy\right)dx+\int_\mathbf{I}\chi^2\left(\int_0^{\epsilon h}g''^2dy\right)dx\\
&= \int_\mathbf{I} \left[\chi''+\chi\left(\frac{\pi}{\epsilon h}\right)^2\right]^2dx-2\int_\mathbf{I}\left(\frac{\pi^2}{3}+\frac{1}{4}\right)\frac{h'^2}{h^2}\chi\left[\chi''+\chi\left(\frac{\pi}{\epsilon h}\right)^2\right]dx+\int_\mathbf{I}4\left(\frac{\pi^2}{3}+\frac{1}{4}\right)\frac{h'^2}{h^2}\chi'^2dx\\
&\quad +\int_\mathbf{I} 4\chi\chi'\left[\frac{h'}{h}\frac{h''}{h}\left(\frac{\pi^2}{3}+\frac{1}{4}\right)+\left(\frac{h'}{h}\right)^3\left(-\frac{1}{4}-\frac{4}{3}\pi^2\right)\right]dx\\
&\quad +\int_\mathbf{I} \chi^2\left[\left(\frac{1}{2}+\frac{53\pi^2}{12}+\frac{\pi^4}{5}\right)\left(\frac{h'}{h}\right)^4-\left(\frac{1}{2}+\frac{8\pi^2}{3}\right)\left(\frac{h'}{h}\right)^2\frac{h''}{h}+\left(\frac{\pi^2}{3}+\frac{1}{4}\right)\left(\frac{h''}{h}\right)^2\right]dx\\
\end{align*}\endgroup
\end{enumerate}
Thus 
\begingroup\makeatletter\def\f@size{9}\check@mathfonts
\def\maketag@@@#1{\hbox{\m@th\large\normalfont#1}}%
\begin{align*}
||A_{21}f(x,y)||^2&=||Q\Delta f(x,y)||^2=||\Delta f||^2-||P\Delta f||^2\\
&=\int_\mathbf{I}4\left(\frac{\pi^2}{3}+\frac{1}{4}\right)\frac{h'^2}{h^2}\chi'^2dx +\int_\mathbf{I} 4\chi\chi'\left[\frac{h'}{h}\frac{h''}{h}\left(\frac{\pi^2}{3}+\frac{1}{4}\right)+\left(\frac{h'}{h}\right)^3\left(-\frac{1}{4}-\frac{4}{3}\pi^2\right)\right]dx\\
&\quad +\int_\mathbf{I} \chi^2\left[\left(\frac{1}{2}+\frac{53\pi^2}{12}+\frac{\pi^4}{5}\right)\left(\frac{h'}{h}\right)^4-\left(\frac{1}{2}+\frac{8\pi^2}{3}\right)\left(\frac{h'}{h}\right)^2\frac{h''}{h}
\right]dx\\
&\quad +\int_\mathbf{I}\chi^2\left[\left(\frac{\pi^2}{3}+\frac{1}{4}\right)\left(\frac{h''}{h}\right)^2-\left(\frac{\pi^2}{3}+\frac{1}{4}\right)^2\left(\frac{h'}{h}\right)^4+\left(\frac{1}{16}+\frac{1}{12}\pi^2\right)\left(\frac{h'}{h}\right)^3\right]dx
\end{align*}\endgroup
where 
\begin{align*}
g'&=\frac{\partial}{\partial x}g(x,y)=-\frac{1}{2}\cdot\frac{h'}{h}g-\frac{h'}{h}\frac{\pi y}{\epsilon h}\sqrt{\frac{2}{\epsilon h(x)}}\cos(\frac{\pi y}{\epsilon h(x)})\\
g''&=\frac{\partial^2}{\partial x^2}g(x,y)\\
&=-\frac{1}{2}\left(\frac{h''}{h}-\frac{h'^2}{h^2}\right)g-\frac{1}{2}\frac{h'}{h}g'-\frac{\pi y}{\epsilon h}\left(\frac{h''}{h}-\frac{5}{2}\frac{h'^2}{h^2}\right)\sqrt{\frac{2}{\epsilon h(x)}}\cos(\frac{\pi y}{\epsilon h(x)})\\
&\quad -\frac{h'^2}{h^2}\left(\frac{\pi y}{\epsilon h}\right)^2\sqrt{\frac{2}{\epsilon h(x)}}\sin(\frac{\pi y}{\epsilon h(x)})\\
\end{align*}
and 
\begingroup\makeatletter\def\f@size{9}\check@mathfonts
\def\maketag@@@#1{\hbox{\m@th\large\normalfont#1}}%
\begin{align*}
\int_0^{\epsilon h} g^2 dy &=1\\
\int_0^{\epsilon h} g'^2 dy &=\frac{h'^2}{h^2}\left(\frac{1}{4}+\frac{1}{3}\pi^2\right)\\
\int_0^{\epsilon h} g''^2 dy &=\left(\frac{1}{2}+\frac{53\pi^2}{12}+\frac{\pi^4}{5}\right)\left(\frac{h'}{h}\right)^4-\left(\frac{1}{2}+\frac{8\pi^2}{3}\right)\left(\frac{h'}{h}\right)^2\frac{h''}{h}+\left(\frac{\pi^2}{3}+\frac{1}{4}\right)\left(\frac{h''}{h}\right)^2+\left(\frac{1}{16}+\frac{1}{12}\pi^2\right)\left(\frac{h'}{h}\right)^3\\
\int_0^{\epsilon h} gg' dy &=0\\
\int_0^{\epsilon h} gg'' dy &=-\left(\frac{\pi^2}{3}+\frac{1}{4}\right)\frac{h'^2}{h^2}\\
\int_0^{\epsilon h} g'g'' dy &=\frac{h''}{h}\frac{h'}{h}\left(\frac{\pi^2}{3}+\frac{1}{4}\right)+\left(\frac{h'}{h}\right)^3\left(-\frac{1}{4}-\frac{4}{3}\pi^2\right)
\end{align*}\endgroup
\end{proof}

\end{document}